\documentclass[a4paper,10pt]{amsart}
\usepackage[utf8]{inputenc}
\usepackage[english]{babel}
\usepackage{amsmath}
\usepackage{amsfonts}
\usepackage{amssymb}
\usepackage{wasysym}
\usepackage{lmodern}
\usepackage{hyperref}
\usepackage{xcolor}
\usepackage{hyperref}
\hypersetup{breaklinks=true}
\usepackage{todonotes}
\usepackage{mathtools}

\definecolor{darkred}{rgb}{0.5,0,0}
\definecolor{darkgreen}{rgb}{0,0.5,0}
\definecolor{darkblue}{rgb}{0,0,0.5}
\hypersetup{colorlinks,
linkcolor=darkblue,
filecolor=darkgreen,
urlcolor=darkred,
citecolor=darkblue}

\usepackage{geometry}

\numberwithin{equation}{section}

\usepackage{enumitem}
\setlist{nosep}
\setlist{noitemsep}
\setlist{leftmargin=*}

\usepackage[normalem]{ulem}

\usepackage{mathtools}

\DeclarePairedDelimiter\floor{\lfloor}{\rfloor}

\newtheorem{theorem}{Theorem}

\newtheorem{proposition}{Proposition}[section]
\newtheorem{lemma}[proposition]{Lemma}
\newtheorem{corollary}[proposition]{Corollary}

\newtheorem{claim}[proposition]{Claim}

\newtheorem{remark}[proposition]{Remark}

\theoremstyle{definition}
\newtheorem{definition}[proposition]{Definition}

\newcommand{\g}{\mathrm{g}}

\newcommand{\R}{\mathbb{R}}
\renewcommand{\div}{\mathrm{div}}
\renewcommand{\det}{\mathrm{det}}

\newcommand{\mm}{\mathbf{m}}
\newcommand{\tmm}{\widetilde{\mm}}
\newcommand{\mms}{\mm_s}

\newcommand{\tmms}{\widetilde{\mm}_s}

\renewcommand{\rm}{\mathrm{m}}
\newcommand{\trm}{\widetilde{\rm}}
\newcommand{\rms}{\rm_s}

\newcommand{\trms}{\widetilde{\rm}_s}

\newcommand{\hmu}{\mathfrak{h}_0}

\newcommand{\fze}{\varphi_{z, \epsilon}}

\newcommand{\KNbeta}{\mathrm{K}_{N}^{\beta}}
\newcommand{\id}{\mathrm{Id}}
\newcommand{\Dd}{\mathsf{D}}
\renewcommand{\phi}{\varphi}
\renewcommand{\epsilon}{\varepsilon}

\newcommand{\D}{\mathrm{D}}

\newcommand{\1}{\mathsf{1}}
\newcommand{\2}{\mathsf{2}}

\newcommand{\hal}{\frac{1}{2}}

\newcommand{\XN}{\mathrm{X}_N}

\newcommand{\Esp}{\mathbb{E}}
\renewcommand{\P}{\mathbb{P}}

\newcommand{\F}{\mathsf{F}}

\newcommand{\Pot}{\mathrm{Pot}_N}

\newcommand{\kk}{\mathsf{k}}

\newcommand{\Ani}{\mathsf{Ani}}

\newcommand{\rr}{\mathsf{r}}

\newcommand{\LN}{\mathbf{L}_N}
\newcommand{\bXN}{\mathbf{X}_N}

\renewcommand{\d}{\mathrm{d}}

\newcommand{\diagc}{\left(\R^2 \times \R^2\right) \setminus \triangle}

\newcommand{\logz}{\log_z}

\newcommand{\Z}{\mathbb{Z}}

\newcommand{\0}{\mathsf{0}}

\newcommand{\Latt}{\Lambda_\delta}
\newcommand{\Ddr}{\Dd_{r}}

\newcommand{\TDOCP}{\textrm{2DOCP}}

\newcommand{\dist}{\mathrm{dist}}

\newcommand{\EE}{\mathcal{E}}
\newcommand{\DD}{\mathrm{D}}
\newcommand{\Cc}{\mathtt{C}}

\newcommand{\EnerPts}{\mathsf{EnerPts}}
\newcommand{\Pbeta}{\mathbb{P}^\beta}
\newcommand{\PNbetamu}{\Pbeta_{N, \mm}}
\renewcommand{\AA}{\mathsf{A}_{\mathsf{1}}}
\newcommand{\PNbeta}{\Pbeta_N}
\newcommand{\ZNbeta}{\mathrm{Z}_{N, \beta}}
\newcommand{\LNmu}{\LN^{\mm}}
\newcommand{\ellA}{\ell_a}
\newcommand{\ellB}{\ell_b}
\newcommand{\fA}{f_a}
\newcommand{\fB}{f_b}
\newcommand{\zA}{z_a}
\newcommand{\zB}{z_b}

\newcommand{\sstar}{s_\star}
\newcommand{\PNep}{\mathrm{Pot}_{N, \epsilon}}

\newcommand{\elld}{\ell_\delta}

\newcommand{\logzp}{\log_{z'}}
\newcommand{\rhoep}{\rho_\epsilon}
\newcommand{\tlog}{\widetilde{\log}}
\newcommand{\Gzz}{G^{z,z'}}
\newcommand{\trhob}{\widetilde{\rho}_\beta}
\newcommand{\tG}{\widetilde{G}^{z,z'}}
\newcommand{\rmsi}{\rms^{(i)}}
\newcommand{\mmsi}{\mms^{(i)}}
\newcommand{\rmsip}{\rms^{(i+1)}}
\newcommand{\mmsip}{\mms^{(i+1)}}
\newcommand{\fii}{f^{(i)}}
\newcommand{\Lell}{\mathcal{L}_{z,\ell}}

\newcommand{\N}{\mathbb{N}}

\newcommand{\HN}{\mathsf{H}_N}

\newcommand{\corO}{}
\newcommand{\corT}{}
\newcommand{\corG}{}

\author{Gaultier Lambert}
\address[Gaultier Lambert]{KTH Royal Institute of technology, Matematik, Lindstedtsv\"agen 25, 11428 Stockholm}
\thanks{G.L. acknowledges the supports of the Ambizione grant S-71114-05-01 from the Swiss National Science Foundation and of the starting grant 2022-04882 from the  Swedish Research Council.}
\email{glambert@kth.se}

\author{Thomas Lebl\'{e}}
\address[Thomas Lebl\'{e}]{Universit\'{e} de Paris-Cit\'{e}, CNRS, MAP5 UMR 8145, F-75006 Paris, France.}
\thanks{T.L. acknowledges the support of JCJC grant ANR-21-CE40-0009 from Agence Nationale de la Recherche.}
\email{thomas.leble@math.cnrs.fr}

\author{Ofer Zeitouni}
\address[Ofer Zeitouni]{Department of Mathematics, Weizmann Institute of Science, Rehovot 76100, Israel.}
\thanks{The third author was partially supported by Israel Science Foundation grant number 421/20.}
\email{ofer.zeitouni@weizmann.ac.il}
\email{}

\title[LLN for the maximum of the 2D Coulomb gas potential]{Law of large numbers for the maximum of the two-dimensional Coulomb gas potential}
\begin{document}
\date{}

\begin{abstract}
We derive the leading order asymptotics of the logarithmic potential
of a two dimensional Coulomb gas at arbitrary positive temperature. The proof is based on precise evaluation
of exponential moments, and the theory of Gaussian multiplicative chaos.
\end{abstract}
\maketitle

\section{Introduction}

\subsection{Setting and main result}
We are interested in proving a law of large numbers for the maximal value of the random electrostatic (or logarithmic) potential generated by the particles of a two-dimensional Coulomb gas - sometimes also called a 2d log-gas, or ``two-dimensional, one-component plasma'' (\TDOCP).

\subsubsection*{The \TDOCP}
\corT{Let $N \geq 1$ and let $\XN := (x_1, \dots, x_N)$ be a $N$-tuple of (distinct) points in $\R^2$ and let
\begin{equation}
\label{eqHN}
\HN(\XN) := \hal \sum_{1 \leq i \neq j \leq N} - \log |x_i - x_j| + N \sum_{i=1}^N \frac{|x_i|^2}{2},
\end{equation}
be the “energy” of $\XN$, given by the sum of all the pairwise \emph{logarithmic interactions} between points plus the effect of the so-called \emph{quadratic confining potential} $x \mapsto \frac{|x|^2}{2}$ on each particle. \\ Let $\beta>0$ be fixed. We will work with the probability measure $\PNbeta$ on $\left(\R^2\right)^N$ whose density is defined as:
\begin{equation}
\label{eq:Pnbetav2}
\d\PNbeta(\XN) := \frac{1}{\ZNbeta} \exp\left( - \beta \left( \HN(\XN) \right) \right) \d\XN,
\end{equation}
where $\d\XN := \d x_1 \dots \d x_N$ is the Lebesgue measure on $(\R^2)^N$ and $\ZNbeta$ is the normalizing constant, or \emph{partition function}, namely the following integral:
\begin{equation*}
\int_{\left(\R^2\right)^N} \exp\left( - \beta \left( \HN(\XN) \right) \right) \d\XN.
\end{equation*}
The measure $\PNbeta$ is the \textit{canonical Gibbs measure} of a \TDOCP \ at \textit{inverse temperature} $\beta$ and with quadratic confinement, cf. e.g. \cite[Chap.~15]{Forrester}.}

\vspace{0.2cm}

Henceforth, we fix an arbitrary value of $\beta > 0$ and $N \geq 2$. We let $\XN$ be a random variable in $(\R^2)^N$ distributed according to the Gibbs measure $\PNbeta$, and let $\bXN$ be the associated random point measure on $\R^2$ as defined above. Unless specified otherwise, expectations in what follows are taken under $\PNbeta$.

\subsubsection*{Main result}
In the sequel, $\Dd(z,r)$ denotes the closed disk of center $z$ and radius $r$ in $\R^2$, we set $\Dd_r=\Dd(0,r)$ and $\Dd=\Dd_1$.

\corT{Let $\mm_0$ be the probability measure with uniform density, denoted $\rm_0$, on the unit disk $\Dd$. As we recall below in Section \ref{sec:rewriting}, under $\PNbeta$ the points $(x_1, \dots, x_N)$ tend to arrange themselves at the macroscopic level according to the so-called \emph{equilibrium (or background) measure} $\mm_0$.}

Define the \emph{Coulomb gas potential} generated by $\XN = (x_1, \dots, x_N)$ (and the background measure) as the following random \corG{real-valued} field on $\R^2$:
\begin{equation}
\label{def:PotXN}
\Pot : z \mapsto \sum_{i=1}^N \log |z-x_i| - N \int_{\Dd} \log |z -x| \d \mm_0(x).
\end{equation}
In physical terms, $\Pot(z)$ corresponds to the value at $z$ of the electrostatic potential generated by the system of charges and the background measure $\mm_0$. Obviously $\Pot(z)$ is equal to $-\infty$ whenever $z$ coincides with one of the point charges. Our main result is the following description of the \textit{maximum} of $\Pot$ over closed disks in the interior of~$\Dd$:

\begin{theorem}[LLN for the max of the 2D Coulomb gas potential]
\label{theo:main}
For all $r \in (0,1)$ we have:
\begin{equation}
\frac{1}{\log N} \max_{z \in  \Dd(0,r)} \Pot(z) \longrightarrow \frac{1}{\sqrt{\beta}}\text{ as } {N \to \infty}, \text{ in probability.}
\end{equation}
\end{theorem}
As a byproduct of our proof, we also obtain a control on a certain regularization of $\Pot$ at microscopic scale (see Corollary~\ref{cor:UB}), which allows us to state a uniform control on the fluctuations of linear statistics for a certain class of $C^2$ test functions (see Proposition~\ref{prop:unibound}).

\subsection{Connections with the literature}
Our interest in Theorem \ref{theo:main} is motivated by connections with the theory of random matrices and the theory of logarithmically correlated fields.

\subsubsection*{The Ginibre ensemble and normal matrix models}
For the specific value $\beta =2$ of the inverse temperature in \eqref{eq:Pnbetav2}, $\PNbeta$ \corT{coincides with} the joint distribution of eigenvalues for \textit{Ginibre} matrices (i.e. matrices whose entries are i.i.d. complex Gaussians normalized by $1/\sqrt{N}$) see again \cite[Chap.~15]{Forrester} or the recent survey \cite{Forrester_review}. \corT{In particular,} the Coulomb gas potential $z \mapsto \Pot(z)$ in \eqref{def:PotXN} can then be interpreted as a \corT{\emph{log-characteristic polynomial}, namely the logarithm of the absolute value of the characteristic polynomial of the associated complex Ginibre matrix}. \corT{In this ($\beta = 2$) case, the statement of Theorem~\ref{theo:main} was proved in \cite{Lambert_2020}.} \corO{We also mention \cite{WW}, where exponential moments of $\Pot(z)$ are computed when $\beta=2$; from the results of \cite{WW}, the upper
bound in Theorem \ref{theo:main} can be readily obtained in that case, but even for $\beta=2$, mixed moments at different points
need to be computed in order to obtain the lower bound.}

\vspace{0.2cm}


\corT{The Ginibre ensemble has the property of being \emph{determinantal} (see e.g. \cite[Chap.~6.4]{hough2009zeros}), which is akin to being ``integrable'' and gives access to hard but explicit computations. Note that one can consider more general \emph{normal matrix models} for which the inverse temperature is still $\beta =2$ but the ``confining'' potential $\frac{|x|^2}{2}$ in \eqref{eqHN} is replaced by other potentials, see e.g. \cite{ameur2011fluctuations,Ameur_2015,Ameur_2021}. These models are also determinantal and, in contrast, the present paper fits into a line of work exploring what can be done \emph{without} the determinantal structure.}

\subsubsection*{Logarithmically correlated structure}
\corT{For the Ginibre ensemble, following a prediction of Forrester and using determinantal techniques, \cite{RV07} proved that the log-characteristic polynomial $\Pot$ converges to some planar Gaussian Free Field (GFF). This follows from \emph{a central limit theorem (CLT) for fluctuations of linear statistics} i.e.~quantities of the form $\sum_{i=1}^N f(x_i) - N \int f \d \mm_0$, with $f$ of class~$C^1$. For normal matrix models, similar questions are treated (among other things) in \cite{ameur2011fluctuations,Ameur_2015,Ameur_2021}.
The CLT was extended to arbitrary temperatures $\beta > 0$ in \cite{bauerschmidt2019two,LebSerCLT} with a slightly stronger assumption on the regularity of $f$. It implies again that for all $\beta > 0$ the potential $\Pot$ converges, in some weak sense, to a two-dimensional GFF.}
\corT{The two-dimensional GFF is a well-known example of a 
\corO{Gaussian logarithmically correlated field, that is, a centered distribution-valued  
Gaussian field whose correlation decays with the logarithm of the distance.
Such fields can be seen as possessing essentially independent contributions from
different dyadic scales, see the discussion in \cite{biskup}.} 
A question of particular interest is the study of extreme values of such fields,
\corO{which has been treated e.g.~in \cite{DRZ17} in some generality.} In view of the convergence results mentioned above, it is thus natural to ask whether the maximum of $\Pot$ behaves like the maximum of a 2d logarithmically correlated field, and Theorem \ref{theo:main} shows that this is true \emph{for the leading order} (see Section \ref{sec:open} for a discussion about lower order terms).}

\subsubsection*{{One-dimensional analogues.}}
\corT{The density \eqref{eq:Pnbetav2} restricted to the real line (resp. the unit circle) coincides with the eigenvalue distribution of the so-called Gaussian (resp. Circular) $\beta$-Ensemble (G$\beta$E, resp. C$\beta$E), which are Hermitian (resp. unitary) random matrix models - the cases $\beta = 2$ and (to a slightly lesser extent) $\beta =1, 4$ being of particular interest.}

Extremes of the (logarithm of the) characteristic polynomial for these matrix models have generated much interest, especially in the case of the CUE (corresponding to the C$\beta$E with $\beta=2$), due to a celebrated
conjecture of Fyodorov, Hiary and Keating \cite{FHK12} that predicts the limiting form of the fluctuations of the maximum and links these to analogous fluctuations for the Riemann zeta-function.

This has stimulated much recent work, starting with \cite{ABB}, \cite{PaquetteZeitouni02} (for the CUE), \cite{CMN} and \cite{PZarXiv} (for the C$\beta$E); the last article indeed proves a version of the FHK conjecture.
\corO{On a related subject, recall that a 
Gaussian multiplicative chaos (GMC) on the unit circle is the weak limit of measures whose densities with respect to the Lebesgue measure are the exponential of a smoothed version
of a logarithmically correlated Gaussian field, properly normalized (see \cite{RV14} for background). It has been established that powers of the 
CUE characteristic polynomial,
viewed as
the density of a measure with respect to the Lebesgue measure 
on the unit circle,
converge to a GMC} in the so-called $L^1$ phase 
\cite{NWS20}. Such measures are limits of exponentials of smoothed logarithmically correlated fields and they describe the fluctuations of thick points (extreme level sets) of the characteristic polynomial \cite{JLW22}.
Weaker analogous results for the case of GUE are contained in \cite{LambertPaquette01} and \cite{CFLW21}, see also \cite{BMP22} and \cite{ABZ22} for the G$\beta$E.

\corT{\subsubsection*{Connection with the Quantum Hall Effect}
The probability density $\PNbeta$ is connected to the study of the (fractional) Quantum Hall Effect (QHE) through the so-called \emph{Laughlin wave function} (more precisely, its absolute square). There is a huge literature on QHE, let us simply refer to the expository text \cite{rougerie2019laughlin} which presents the connection with Coulomb gases using a terminology very close to ours. As we will explain next, our proof of Theorem \ref{theo:main} relies on computing the asymptotics of (joint) exponential moments of $\Pot$. These asymptotics are also related to the problem of determining the statistics of certain types of quantum quasi-particles arising in fractional QHE experiments, by considering  modifications of the Laughlin function as trial states. This connection is explained in further details in \cite[Section~1.3]{LLR22} - note however that the estimates required for that problem go beyond the precision achieved here.}

\subsection{Comments on the strategy.}

\subsubsection*{Existing strategies for logarithmically correlated fields}

\corO{As noted above,
a} general methodology to handle extremes of Gaussian logarithmically correlated fields has been developed in the last decades; due to space limitations, we do not review here the history, and refer instead to \cite{DRZ17} and \cite{biskup} for details. When applied outside the Gaussian context, this methodology 
requires the \corO{decomposition of the field to (essentially) independent 
contributions from different dyadic scales,
the evaluation of exponential moments of the field,
together with the introduction of certain barriers, that is restrictions
on the partial sums of these contributions; we refer again to
\cite{DRZ17} for precise definitions.}
These techniques seem crucial in obtaining sharp results (at the level of $O(1)$ fluctuations for the extremes), and are often hard to implement outside the Gaussian setup. For example, the works mentioned above concerning the extremes of the electrostatic potential for log-gases on the unit circle (\corT{which, for finite $N$, are \emph{not} Gaussian fields}) use, at different levels of precision, variants of these methods, with much technical work going 
into \corO{obtaining decomposition of the fields, inserting appropriate barriers, and controlling comparisons with the Gaussian setup.}

\corT{In this paper, we crucially rely on an} observation made in \cite{LOS18} and \cite{CFLW21} using \corO{GMC theory
in order to obtain the leading order of fluctuations; one can bypass the
use of barriers and sharp results on asymptotic independence often obtained by computing characteristic functions, at the cost of obtaining sharp estimates on exponential moments \corO{of $\Pot$}.}

\subsubsection*{Related computations of exponential moments}
In the context of log-gases, evaluation of exponential moments is at the heart of many proofs of the central limit theorem for linear statistics. \corT{An early application of this method is due to Johansson \cite{johansson1998fluctuations} for eigenvalues of random Hermitian matrices, which can be mapped to certain one-dimensional log-gases (at $\beta = 2$). \corO{He then proceeded to
analyse the log-gas model for all $\beta$. This was later refined in
e.g. \cite{Shcherbina_2014}}.} For random normal matrix models, a proof of the CLT for linear statistics going along the same lines is sketched in \cite[Sec.~7.2]{ameur2011fluctuations}. For the two-dimensional case at arbitrary temperature, the method inspired by Johansson was implemented in \cite{LebSerCLT,bauerschmidt2019two,serfaty2020gaussian}, with additional analytic challenges compared to the $1d$ case.
\corG{Our analysis is based on this approach that we review below.}

\corT{\subsubsection*{The electric energy approach}
The analysis of the Coulomb gas used here (Section \ref{sec:prelim} and Section \ref{sec-UB} as well as the Appendix) relies on the general ``electric energy'' approach to 2DOCP's as developed by Serfaty and co-authors, starting with \cite{sandier20152d}. In particular we use:
\begin{itemize}
\item The “splitting formula” of Sandier-Serfaty and the general idea of working with the equilibrium measure $\mm_0$ instead of the confining potential, which involves a slight re-writing of the Gibbs measure $\PNbeta$, bringing us closer to the physics point of view on 2DOCP's, see Section \ref{sec:rewriting}.
\item The electric energy and its local versions, together with local laws i.e. good controls on the exponential moments of the local energy \emph{up to the microscopic scale} (which is crucial for us) as in \cite{armstrong2019local}, see Section \ref{sec:Energy}.
\item The general spirit of controlling fluctuations through the electric energy, see Section \ref{sec:fluctuations}.
\item The ``transportation'' approach introduced in \cite{LebSerCLT} and the fine energy expansion along a transport found in \cite{serfaty2020gaussian}.
\item Computation of exponential moments of linear statistics, relative expansion of partition functions and Serfaty's ``smallness of anisotropy'' trick, for which up-to-date statements are found in \cite{serfaty2020gaussian}.
\end{itemize}
The last two items are the heaviest technically,
and we postpone their detailed discussion to the Appendix.}

\subsection{Sketch of the proof}
The proof of Theorem \ref{theo:main} is split into an upper bound on the typical maximal value of the electrostatic potential together with a matching lower bound. \corT{In both parts, the goal can be understood as making a comparison to an ideal ``Gaussian case'' where the values of $\Pot$ would be} \corG{a Gaussian logarithmically correlated field.
However, one can ignore the correlations for the upper-bound.}

\corT{\subsubsection*{Upper bound}
The core of the proof for the upper bound (in Section \ref{sec-UB}) is a good control of exponential moments for the values of $\Pot$ i.e. for linear statistics of the form
\begin{equation*}
\sum_{i=1}^N \log|z-x_i| - N \int \log |z-x| \d \mm_0(x), \quad z \in \Dd_1.
\end{equation*} \\
In general, moments of linear statistics can be controlled either by purely energy-based considerations (see Lemma~\ref{lem_fluct_uniform}) or in a more precise fashion using the results of \cite{serfaty2020gaussian}. The first option is fairly robust but it usually yields sub-optimal estimates, so we would like to use the second option, which is well-suited to linear statistics of smooth, compactly supported functions living at a certain ``scale''. However, when considering the test function $x \mapsto \log|z - x|$ (for $z \in \DD_1$) several problems arise:
\begin{enumerate}
\item It is inherently multi-scale.
\item It is singular near $x = z$, whereas the results of \cite{serfaty2020gaussian} require a few derivatives.
\item It is not compactly supported. In fact, as an inspection of the proof of \cite{serfaty2020gaussian} reveals, the real issue is not the lack of compact support but rather the fact that \emph{the total mass of its Laplacian is not $0$}.
\end{enumerate}
In order to be able to deal with the first item, i.e.~to treat test functions that live on different scales, we go back to the proof of \cite{serfaty2020gaussian} and make the necessary adaptations. As a first step, we recast the result of \cite{serfaty2020gaussian} in a slightly different way (Proposition \ref{prop:compar1}), which we then use iteratively to treat the multi-scale setting (Proposition \ref{prop:compar2}). This is carried out in Section \ref{sec:prelim}, with proofs postponed to the appendices. \\
To deal with the second and third items, we regularize the test function $x \mapsto \log|z - x|$ and we ``center'' it - namely, we substract some well-chosen test function so that the Laplacian of the difference has total mass $0$.
\begin{itemize}
\item The regularization procedure is standard. By a simple trick (Lemma \ref{lem:regulPot}) using sub-harmonicity, we see that bounding the maximal value of the \emph{regularized} version of $\Pot$ is enough to control the maximal value of the ``true'' $\Pot$.
\item The ``centering'' is constructed ad hoc in Section \ref{subsec-center}, and Proposition \ref{prop:expg} (relying on a bound for exponential moments) guarantees that \emph{the presence of the centering does not matter when evaluating the maximal value of the field}.
\end{itemize}
The work done so far allows us to estimate the exponential moments of linear statistics corresponding to an appropriately regularized version of $x \mapsto \log|z - x|$ for any \emph{fixed} $z$ in the interior of the unit disk. This is only a point-wise bound (recall that we want to control the \emph{maximal} value over all $z$'s) but the probabilistic tails obtained through Markov's inequality are good enough to go through a big union bound and to control the maximal value of the potential over all points of the disk living on a very narrow lattice of stepsize $\approx N^{-1/2-\delta}$ with $\delta$ small, see Lemma \ref{lat-est}. \\
Finally, it remains to extend the result from that lattice to the whole disk. This requires a control of the difference between the potential felt at two points separated by $\approx N^{-1/2-\delta}$, which is presented in Proposition \ref{prop:lattice}.}

\subsubsection*{Lower bound}
The proof of the lower bound is provided in Section \ref{sec-LB}. \corT{As mentioned earlier, it follows the recipe of \cite{CFLW21} and is based on the construction of a sequence of measures obtained from exponentiating regularized versions of $\Pot$ similar to those introduced for the upper bound, see \eqref{eq-mugamma}.}

The heart of the proof consists then in showing (see Proposition \ref{prop:GMC}) that this sequence of measures converges to a Gaussian Multiplicative Chaos (GMC), which is \corT{defined as} the limit of similar objects constructed from a Gaussian process. The result of \cite{CFLW21} guarantees that \corT{in order to prove the desired convergence, it is enough to obtain asymptotics of exponential moments of  linear combinations of  regularized logarithm centered at different points}, see \eqref{asymp_mom}.
The proof of the latter is done by induction and uses in a crucial way the two-scale statement of Proposition~\ref{prop:compar2} (that was already used for proving the upper bound).

\subsection{{Open problems}}
\label{sec:open}
The most obvious open question regarding Theorem \ref{theo:main} is that it only gives the \emph{leading order} of the maximum of $\Pot$. If one wants to actually see the influence of the underlying logarithmically correlated structure, one needs to evaluate (at least) the next order correction, which is expected to be $\frac{3}{4 \sqrt{\beta}} \log\log N (1 + o(1))$ due to the log-correlated structure. The techniques for achieving that go beyond our methods.

{In a different direction, it} is natural to consider replacing $|x|^2$ in \eqref{eqHN} by other growing (real) functions~$f$.
Applying our methods to that case requires three ingredients: first, one needs regularity of the background density $\rm_0$, and to modify
the
electrostatic potential $\hmu$, see \eqref{def:hmu} accordingly. Second, one would need
to modify the background function $\g$ of \eqref{eq:g}, which
is easy to do in the radial case. And third, one should look for
a replacement for Claim \ref{claim_scalin}, whose proof is
based on a scaling argument.
In the case of monomials $f(x)=|x|^q$, it is simple to carry out the adaptations, but already in the case of a general (even) polynomial $f(|x|)$ one needs to find a
replacement for
the scaling argument.

A particular case of interest, concerns the \textit{real} Ginibre ensemble, see \cite{FN07}. There, one needs to deal with the symmetry of the point configuration $\XN$, as well as with the special role of the real axis. This would require significant changes in our derivation, and we leave this as an open problem.

Finally, we expect the result of Theorem \ref{theo:main} to remain true if one considers the maximum of $\Pot$ over the whole unit disk (or even the whole plane).
We also expect that Proposition~\ref{prop:MC} holds without any regularization; see Remark~\ref{rk:MC}.

\subsection{Notation} \label{sec:not}
\begin{itemize}
	\item If $\Omega$ is a measurable subset, $\bXN(\Omega)$ denotes the number of points of $\XN$ contained in $\Omega$.
	\item We denote measures with a bold typeface (e.g.~$\mm$) and their densities with respect to the Lebesgue measure on $\R^2$ with a roman typeface (e.g.~$\rm$).
	\item \corG{For $\kk\ge 0$,  provided $\varphi$ is $C^\kk(\R^2)$, we denote by $|\varphi|_\kk$ {the $L^\infty$-norm of its} $\kk^{th}$ derivative(s).}
\corT{With this notation, if $\varphi_{\ell} := \varphi(\cdot / \ell)$ is the “rescaled” version of $\varphi$ at scale $\ell > 0$ we have $|\varphi_{\ell}|_\kk = |\varphi|_\kk \times \ell^{-\kk}$.}
	\item Integrals with respect to the Lebesgue measure are often written without explicitly mentioning the volume form, ie $\int \varphi =\int \varphi(x) \d x$.
	\item If $A, B$ are two quantities (depending on various parameters) we write $A \preceq B$ {(or $A=O(B)$)} when $|A|$ is bounded by some universal constant times $|B|$. We write $A=o(B)$ or $A\ll B$
	  if $A/B\to 0$ when a (sometimes implicit)
	  parameter goes to infinity. When no parameter is mentioned, it is understood that the parameter is $N$. We write $A \asymp B$ when $A = O(B)$ and $B = O(A)$. \corO{We write $A\gg B$ if $A,B$ are positive and $B\ll A$.}
	\item When $\mm$ is a probability measure on $\R^2$ with continuous density $\rm$, we denote its relative entropy (with respect to Lebesgue) by $\EE(\rm) := \int_{\R^2} \rm \log \rm$.
	\item  \corG{If $\varphi$ is $C(\R^2)$ and $c>0$, by ``$\varphi$ is $c$-Lipschitz'' we mean that $\sup_{x,y\in\R^2,x \neq y} \left|\frac{\varphi(x)-\varphi(y)}{x-y}\right| \leq c$.}
\end{itemize}

\subsection*{Acknowledgment}
\emph{We thank Sylvia Serfaty for communicating early versions of \cite{serfaty2020gaussian} to us and making some statements thereof more easily citable for our purposes. \corO{We thank the anonymous referees for comments that improved 
the presentation of our results, and for a careful reading.}}

\section{Preliminaries}
\label{sec:prelim}

\subsection{Re-writing the energy and the Gibbs measure}
\label{sec:rewriting}
\corT{Recall that $\mm_0$ denotes the uniform \corG{probability} measure on the unit disk. The following result is classical (see e.g. the book \cite{Serfaty_2015}):
\begin{lemma}
As $N \to \infty$, under $\PNbeta$, the empirical measure of the points $\frac{1}{N} \sum_{i=1}^N \delta_{x_i}$ converges weakly to the equilibrium measure $\mm_0$ both in probability (and almost surely when coupling all the $\PNbeta$'s in the trivial way), with large deviations at speed $N^2$.
\end{lemma}
In particular, it makes sense to consider the difference $\frac{1}{N} \sum_{i=1}^N \delta_{x_i} - \mm_0$ as encoding the ``second-order'' behavior \corG{(fluctuations)} of the system. Starting with \cite{sandier20152d}, studies of the 2DOCP have benefited from the following (seemingly simple) rephrasing:
Let $\bXN := \sum_{i=1}^N \delta_{x_i}$ be the purely atomic measure of total mass $N$ on $\R^2$ associated to the $N$-tuple of positions $\XN = (x_1, \dots, x_N)$. We define the \textit{logarithmic interaction energy} $\F(\XN, \mm_0)$ as:
\begin{equation}
\label{def:FXN}
\F(\XN, \mm_0) := \hal \iint_{\diagc} - \log|x-y| (\d\bXN - N \d\mm_0)(x) (\d\bXN - N \d\mm_0)(y),
\end{equation}
where $\triangle$ denotes the diagonal in $\R^2 \times \R^2$. Recalling that $-\log$ is (up to a multiplicative constant) the Coulomb kernel in $\R^2$, we can think of $\F(\XN, \mm_0)$ as being the electrostatic interaction energy of a neutral system made of $N$ point charges placed at $(x_1, \dots, x_N)$ together with a continuous ``neutralizing'' background $N \mm_0$ of opposite charges, and this is indeed the point of view used in the physics literature about 2DOCP's (see e.g. \cite{alastuey1981classical}).\\
We also introduce an auxiliary function $\zeta$ (the ``effective confinement''), which vanishes on $\Dd$ and is set to:
\begin{equation}
\label{def:zeta}
\zeta(x) := - \log |x| + \frac{1}{2} |x|^2 - \frac{1}{2} \text{ on $\R^2 \setminus \Dd$.}
\end{equation}
The \emph{splitting formula} of Sandier-Serfaty (\cite[Sec.~2]{sandier20152d}) consists simply in observing that:
\begin{equation}
\label{SplittingFormula}
\HN(\XN) = \F(\XN, \mm_0) + 2 \sum_{i=1}^N \zeta(x_i) + \text{ a constant term depending on $N$ but not on $\XN$.}
\end{equation}
In particular, the Gibbs measure $\PNbeta$ introduced in \eqref{eq:Pnbetav2} coincides with
\begin{equation}
\label{eq:Pnbetav1}
\d\PNbeta(\XN) := \frac{1}{\KNbeta} \exp\left( - \beta \left( \F(\XN, \mm_0) + 2N \sum_{i=1}^N \zeta(x_i) \right) \right) \d\XN,
\end{equation}
where $\KNbeta$ is the corresponding \emph{partition function}, namely:
\begin{equation}
\label{def:KNbeta}
\KNbeta := \int_{\left(\R^2\right)^N} \exp\left( - \beta \left( \F(\XN, \mm_0) + 2N \sum_{i=1}^N \zeta(x_i) \right) \right) \d\XN.
\end{equation}
In the sequel, we will work with the expression \eqref{eq:Pnbetav1} instead of \eqref{eq:Pnbetav2} for the Gibbs measure of the 2DOCP, and with $\F(\XN, \mm_0)$ instead of $\HN(\XN)$ as the ``energy'' of the system.
}

\begin{remark}
\label{rem:perfectconfinement}
In the physics literature about the \TDOCP, the Gibbs measure is often written down as in \eqref{eq:Pnbetav1} but with an ``effective confinement''
$\zeta$ set to $+\infty$ outside $\Dd$ (``perfect confinement''). Our analysis in the present paper applies to this model as well (the only differences would appear when studying properties close to the boundary, which is not our purpose).
\end{remark}


\subsection{Energy at global and local scales}
\label{sec:Energy}
{Let $\mm$ be a probability measure on $\Dd$ with a density that is continuous and bounded below by a positive constant on $\Dd$.}

\subsubsection*{Global energy}
If $\XN$ is a $N$-tuple of points and $\bXN$ is the associated atomic measure of mass $N$, we {extend the definition \eqref{def:FXN}} and define the ``global energy'' $\F(\XN, \mm)$ as:
\begin{equation}
\label{def:globalenergy}
\F(\XN, \mm) := \hal \iint_{\diagc} - \log|x-y| (\d\bXN - N \d\mm)(x) (\d\bXN - N \d\mm)(y).
\end{equation}
We introduce the associated Gibbs measure (where\footnote{{The function $\zeta$ plays almost no role in our analysis, as we are focused on the bulk of the system. For simplicity we keep “the same $\zeta$” in all cases.}} $\zeta$ is as in \eqref{def:zeta})
\begin{equation}
\label{def:PNbetamu}
\d \PNbetamu(\XN) := \frac{1}{\KNbeta(\mm)} \exp\left(- \beta \left( \F(\XN, \mm) + 2N \sum_{i=1}^N \zeta(x_i) \right) \right) \d \XN,
\end{equation}
with the corresponding partition function:
\begin{equation}
\label{def:KNbetamu}
\KNbeta(\mm) := \int_{(\R^2)^N} \exp\left(- \beta \left( \F(\XN, \mm) + 2N \sum_{i=1}^N \zeta(x_i) \right) \right) \d \XN.
\end{equation}
It is known (see e.g. the pioneering analysis of \cite[Thm. 1]{sandier20152d}) that {for all $\beta > 0$}:
\begin{equation}
\label{global_law}
\log \KNbeta(\mm) = \frac{\beta}{4} N \log N + O_\beta(N).
\end{equation}
\corT{One should thus think of 
$\F(\XN, \mm) + \frac{1}{4} N \log N$ as being the ``interesting'' global energy term, which is random and typically of order $N$. In fact, in the $\frac{1}{4} N \log N$ term, $\log N$ is related to the logarithm of the ``microscopic scale'' (here $N^{-1/2}$) and the $N$ factor is simply the total number of particles. This will be useful to keep in mind when encountering the local version below.}

\subsubsection*{Length scales} \corT{The system is supported on the unit disk, hence the natural \emph{global} scale is $\ell =1$.} For many interesting questions it is crucial to understand the system at \emph{local} scales i.e. in squares (or disks) of size $\ell \ll 1$. One distinguishes between \emph{mesoscopic} scales $\ell$ such that $N^{-1/2} \ll \ell \ll 1$ and the \emph{microscopic} scale $\ell \simeq N^{-1/2}$. {A constant $\rho_\beta \geq 1$ depending only on $\beta$ was introduced in \cite[(1.15)]{armstrong2019local}, it corresponds to the “minimal lengthscale” above which good controls on the energy can be obtained. In this paper when considering a length scale $\ell$ we will always assume that}\footnote{{Since \cite{armstrong2019local} work with a different scaling than us, we need to rescale their $\rho_\beta$ by $N^{-1/2}$.}}  
\begin{equation}
\label{eq-100124}
\corO{\ell \geq \trhob := \rho_\beta N^{-1/2}.}
\end{equation} 
 (Note that since $\rho_\beta \geq 1$, we always have $N \ell^2 \geq 1$.)

\subsubsection*{Local (electric) energy}
\corT{Due to the long-range nature of the logarithmic potential, the proper notion of a \emph{local energy} is a bit subtle. Suitable definitions were given in \cite{leble2017local,armstrong2019local}, and we recall below the relevant definitions and facts.}

\vspace{0.2cm}
\corT{\textbf{1. Electric field.} We let $\Pot^{\bXN, \mm}$ (resp. $\nabla \Pot^{\bXN, \mm}$) be the \textit{true electric potential} (resp. \textit{true electric field}) generated by the ``charged system'' $\bXN-\mm$, namely the map (resp. vector field) defined on $\R^2$ by:
\begin{equation*}
\Pot^{\bXN, \mm}(x) := \int - \log |x-y| \d \left(\bXN - \mm \right)(y),\quad  \nabla \Pot^{\bXN, \mm}(x) = \int - \nabla \log |x-y| \d \left(\bXN - \mm \right)(y).
\end{equation*}
We recall that $-\log$ satisfies $- \Delta (-\log) = 2\pi \delta_0$ on $\R^2$ in the sense of distributions. It is easy to check that the following identity is satisfied in the sense of distributions:
\begin{equation}
\label{PropertyPot}
- \Delta \Pot^{\bXN, \mm} = 2\pi \left(\bXN - \mm \right), \quad \text{ with $\nabla \Pot(z) \to 0$ as $|z| \to \infty$,}
\end{equation}
and that $\nabla \Pot^{\bXN, \mm}$ is \corG{(almost surely)} in $L^p_{\mathrm{loc}}$ for $p < 2$ yet fails to be in $L^2$ around each point charge.}

\vspace{0.2cm}

\textbf{2. Truncations.}
\newcommand{\ff}{\mathsf{f}}
\corT{In order to handle the singularities, one often proceeds to a truncation of the fields near each point charge. For $\eta > 0$ we let $\ff_{\eta}$ be the function:
\begin{equation*}
\ff_{\eta}(x) := \max\left( -\log \frac{|x|}{\eta}, 0\right) = \begin{cases} -\log|x| + \log|\eta| & \text{if } x \leq \eta \\ 0 & \text{if } x \geq \eta \end{cases}.
\end{equation*}
For each $i = 1, \dots, N$ let $\eta_i$ be a \corG{positive} real number. For a fixed choice of $\vec{\eta} := (\eta_1, \dots, \eta_N)$, we let $\nabla \Pot^{\bXN, \mm, \vec{\eta}}$ be the \textit{truncated} electric field given by:
\begin{equation*}
\nabla \Pot^{\bXN, \mm, \vec{\eta}} = \nabla \Pot^{\bXN, \mm} - \sum_{i=1}^N \nabla \ff_{\eta(x)}(\cdot - x).
\end{equation*}
We are thus effectively replacing $-\log|x - \cdot|$ by $-\log \eta$ near each point charge. We refer to \cite[Section 2.2 \& Appendix B.1]{armstrong2019local} or to \cite[Sec 3.1]{serfaty2020gaussian} for more details.}

\corT{For $i = 1, \dots, N$ define the “nearest-neighbor” distance $\rr_i$ as:
\begin{equation}
\label{def:nn_distance}
\rr_i := \frac{1}{4} \min \left( \min_{j \neq i} |x_i-x_j|, N^{-1/2} \right).
\end{equation}
We let $\vec{\rr} := (\rr_1, \dots, \rr_N)$, which can serve as a convenient choice of truncation. Note that the $\rr_i$'s are always smaller than $\frac{1}{4} N^{-1/2}$.}

\vspace{0.2cm}

\corT{\textbf{3. Electric formulation of the energy.}
The following identity (see e.g. \cite[Lemma~2.2.]{armstrong2019local}) can be considered as the starting point of the ``electric energy'' approach:
\begin{equation}
\label{ElectricIdentity}
\F(\bXN,\mm) = \hal \left( \frac{1}{2\pi} \int_{\R^2} |\nabla \Pot^{\bXN, \mm, \vec{\rr}}|^2 + \sum_{i=1}^N \log \rr_i \right)  - \sum_{i=1}^{N} \int_{\Dd(x_i,\rr_i)} \ff_{\rr_i}(t-x_i) \d \mm(t).
\end{equation}
}
\vspace{0.05cm}

\corT{To summarize where we stand: the interaction energy between the particles plus the effect of the potential, or equivalently (by \eqref{SplittingFormula}) the electrostatic interaction energy $\F(\bXN, \mm)$ of the charged system ``point charges minus background'' can be (by \eqref{ElectricIdentity}) rephrased as a certain ``electric energy'';
\begin{enumerate}
\item The square of the $L^2$ norm of the corresponding electric field $\int_{\R^2} |\nabla \Pot^{\bXN, \mm, \vec{\rr}}|^2$ after a suitable truncation (this is the role of the $r_i's$). This term is typically of order $N$.
\item  A ``renormalization'' term $\sum_{i=1}^N \log \rr_i$. This term is equal to the constant $-\hal N \log N$ plus something (typically) of order $N$.
\item A correction term $\sum_{i=1}^{N} \int_{\Dd(x_i,\rr_i)} \ff_{\rr_i}(t-x_i) \d \mm(t)$ which is deterministic and bounded \corO{uniformly in $N$ 
 (with constant} depending on $\mm$).
\end{enumerate}
}
\vspace{0.2cm}

\corT{\textbf{4. Local energy.}
Going back to the initial definition \eqref{eqHN} of the energy $\HN(\XN)$, it seems that the natural notion of ``the energy of $\XN$ within a subset $\Omega$'' could be to look at:
\begin{equation*}
\hal \sum_{i \neq j, x_i \in \Omega, x_j \in \Omega} - \log|x_i - x_j| + \sum_{i, x_i \in \Omega} \frac{|x_i|^2}{2}.
\end{equation*}
However, due to the long-range character of the logarithmic interaction, this does not  ``work'' and the electric field introduced above turns out to be a more convenient object to deal with. In short, it is better to localize the $L^2$ norm of the electric field (the field itself remains a \emph{global} object depending on the entire point configuration) than to localize the interaction.}

\corT{We now go through some definitions found in the beginning of \cite[Sec.~2.3]{armstrong2019local}, which will lead us to introduce a proper notion of ``local energy''. Let $\Omega$ be a disk or a square in $\R^2$ and let us take $U = \R^2$ when reading \cite{armstrong2019local}. \corG{In our case,} the function $\mathsf{h}$ defined in \cite[(2.20)]{armstrong2019local} is identically $0$ and we can ignore it. The potential $u$ defined in \cite[(2.22)]{armstrong2019local} is nothing but $\Pot$ (as the equation that needs to be solved is exactly \eqref{PropertyPot}) up to some irrelevant additive constant. The distance $\tilde{\rr}_i$ introduced in \cite[(2.23)]{armstrong2019local} (recall that we have $U = \R^2$ so $\partial U = \emptyset$) becomes here simply:
\begin{equation*}
\tilde{\rr}_i = \begin{cases} \rr_i & \text{ if } \dist(x_i, \partial \Omega) \geq \hal \\
\frac{1}{4} N^{-1/2} & \text{ otherwise.}
\end{cases}
\end{equation*}
With a slight abuse of notation we let $\tilde{\rr}$ be the vector $(\tilde{\rr}_1, \dots, \tilde{\rr}_N)$. The ``local energy in $\Omega$'' as defined in \cite[(2.24)]{armstrong2019local} reads:
\begin{equation}
\label{eq:LocalEnergy}
\F^{\Omega}(\XN) := \frac{1}{2} \left( \frac{1}{2\pi} \int_{\Omega} |\nabla \Pot^{\bXN, \mm, \tilde{\rr}}|^2 + \sum_{i, x_i \in \Omega} \log \tilde{\rr}_i \right)  - \sum_{i, x_i \in \Omega} \int_{\Dd(x_i,\tilde{\rr}_i)} \ff_{\tilde{\rr}_i}(t-x_i) \d \mm(t),
\end{equation}
compare with the right-hand side of \eqref{ElectricIdentity}. One should have in mind (see below for a precise result) that:
\begin{enumerate}
\item $ \int_{\Omega} |\nabla \Pot^{\bXN, \mm, \tilde{\rr}}|^2$ is typically of size $N \times |\Omega|$.
\item $\sum_{i, x_i \in \Omega} \log \tilde{\rr}_i$ is equal to $- \hal n \log N$ (with $n$ the number of points in $\Omega$) plus a term which is typically of order $N \times |\Omega|$.
\item The last term is again a deterministic, bounded correction.
\end{enumerate}
}

\vspace{0.2cm}

\corT{\textbf{5. The local ``energy-points'' density.}
It turns out that many relevant error terms can be controlled by the sum of:
\begin{enumerate}
\item The non-negative part of the local energy \eqref{eq:LocalEnergy} corresponding to the square of $L^2$ norm of the (truncated) electric field in a domain $\Omega$,
\item and the number of points in $\Omega$.
\end{enumerate}
So if $z$ is a point in $\Dd$ and $\ell$ a length scale we will denote by $\EnerPts(z,\ell)$ the quantity:
\begin{equation}
\label{def:EnerPts}
\EnerPts(z,\ell)(\XN) := \int_{\Dd(z, \ell)} |\nabla \Pot^{\bXN, \mm, \tilde{\rr}}|^2  + n, \quad n = \bXN \left(\Dd(z, \ell)\right).
\end{equation}
}

\subsection*{Local laws}
\corO{Recall the condition \eqref{eq-100124}. Due to a weaker understanding of the system near the boundary $\partial \Dd$, one needs to introduce the following additional} condition on $(z, \ell)$ (where $z$ is a point of $\Dd$ and $\ell$ is a length scale):
\begin{equation}
\label{condiell}
\dist\left( \Dd(z, \ell), \partial \Dd\right) \geq \Cc_\beta N^{-1/4},
\end{equation}
where $\Cc_\beta$ is some constant depending only on $\beta$ introduced in the assumptions of \cite[Thm.1]{armstrong2019local}. Then \cite{armstrong2019local} prove the following “local laws”:
\begin{lemma}[Local laws for $\EnerPts$]
If $(z, \ell)$ satisfies \eqref{condiell} and the parameter $t$ is smaller than some constant depending only on $\beta$, then:
\begin{equation}
\label{loc_laws}
\log \Esp_{\PNbetamu} \left[ \exp\left(t \EnerPts(z, \ell) \right) \right] = O\left(t N \ell^2\right),
\end{equation}
with an implicit constant depending only on $\beta$.
\end{lemma}
\corT{Thus in view of the definition \eqref{def:EnerPts}, it means that in $\Dd(z, \ell)$ both the local number of points $n$ and the positive part of the local energy (up to the constant $n \log N$ term) are of order $N\ell^2$ in exponential moments}, and so \emph{down to the microscopic scale} $\ell \simeq N^{-1/2}$ (a crucial improvement over the local laws of \cite{leble2017local} which covered all mesoscopic scales).

\vspace{0.2cm}
\corT{\textbf{How to read \eqref{loc_laws} from the literature.} The statement of \cite[Thm.~1]{armstrong2019local} treats the electric energy and the number of points separately, is formulated in terms of the local energy \eqref{eq:LocalEnergy} and not of its positive part only (see \eqref{def:EnerPts}), moreover the authors work with a different scaling than ours. We now explain how to get the statement that we want to use later, namely the estimate \eqref{loc_laws}, from the literature.\\
\textit{1. Number of points.} Recall that we write here $n$ for the number of points $\bXN \left(\Dd(z, \ell) \right)$. In \cite[Theorem~1]{armstrong2019local} there are two statements about $n$, here we can use \cite[(1.18)]{armstrong2019local} which controls exponential moments of the \emph{discrepancy} (which corresponds to $n - \pi {N}\ell^2$). It is not hard to see that it implies (and is in fact a lot stronger than) a bound on the number of points of the form:
\begin{equation}
\label{nbest}
\log \Esp_{\PNbetamu} \left[ \exp\left(t n\right) \right] = O\left(t N \ell^2\right)
\end{equation}
(under the condition \eqref{condiell} and for $t$ less than a constant depending only on $\beta$).\\
\textit{2. Local energy.} The statement \cite[(1.17)]{armstrong2019local} involves the local energy, whose definition was recalled in \eqref{eq:LocalEnergy}, whereas in $\EnerPts$ we only consider the positive part of it. Moreover there is a  difference in the scaling convention. So (with our apologies to the reader) it might be easier to read the corresponding statements in \cite[Sec. 3]{serfaty2020gaussian} (which uses the same scaling convention as we do) namely:
\begin{enumerate}
\item Control on the local energy with an additive constant term due to rescaling:
\begin{equation*}
\log \Esp_{\PNbetamu} \left[ \exp\left(t \left(\F^{\Dd(z, \ell)}(\XN) + \frac{n}{4} \log N\right) \right) \right] = O\left(t N \ell^2\right),
\end{equation*}
which is \cite[Prop. 3.5, (1)]{serfaty2020gaussian}.
\item Control on the positive part of the local energy in terms of the full local energy:
\begin{equation*}
\int_{\Dd(z, \ell)} |\nabla \Pot^{\bXN, \mm, \tilde{\rr}}|^2 \leq 8 \pi \left( \F^{\Dd(z, \ell)}(\XN) + \frac{n}{4} \log N\right) + Cn,
\end{equation*}
which is \cite[(3.25)]{serfaty2020gaussian}.
\end{enumerate}
Combining these two statements with \eqref{nbest} allows us (in view of the very definition of $\EnerPts$ in \eqref{def:EnerPts}) to derive \eqref{loc_laws} as desired.
}

\subsection{Fluctuations}
\label{sec:fluctuations}
\corG{Recall that $\mm$ is a probability measure on $\Dd$.}
If $\phi$ is a $\mm$-integrable function, we define ``the fluctuation $\LN^{\mm}(\phi)$ of the linear statistics associated to $\phi$ {for the configuration $\XN$'',}  as:
\begin{equation}
\label{def:fluctuations} \LN^{\mm}(\phi) := \sum_{i=1}^N \phi(x_i) - N \int_{\Dd} \phi(x) \d \mm(x).
\end{equation}
{We write $\LN(\varphi)$ for $\LN^{\mm_0}(\varphi)$.} If $\varphi$ is Lipschitz, one can always use the following  non-optimal but \emph{uniform} control on $\LN(\varphi)$. 
\begin{lemma}
\label{lem_fluct_uniform}
Let $z$ be a point of $\Dd$ and $\ell$ be a length scale such that $(z,\ell)$ satisfies \eqref{condiell}. Denote by $\Lell$ the set of all functions that are $\ell^{-1}$-Lipschitz and compactly supported on $\Dd(z, \ell)$. Then for all $t$ such that $|t|$ is smaller than some constant depending on $\beta$, we have:
\begin{equation}
\label{eq:}
\log \Esp \left[ \sup_{\varphi \in \Lell} \exp\left(t \LN(\varphi) \right) \right] = O\left( t \left(N \ell^2\right)^\hal \right),
\end{equation}
with an implicit constant depending only on $\beta$.
A similar estimate holds for $\mathcal{L}$, the set of all 1-Lipschitz function with compact support in  $\Dd(0,2)$.
\end{lemma}
\begin{proof}
This can be traced back to the analysis of \cite[Lemma 3.9]{sandier20152d}, see also \cite[Thm.~5]{rougerie2016higher}. The basic idea is that we have a configuration-wise bound of the form: \corG{for $\phi\in\Lell$,}
\begin{equation*}
|\LN(\varphi)| \leq |\varphi|_{\1} \times \ell \times \EnerPts(z, \ell)^{\hal} \leq \frac12 \left( {\frac{\EnerPts(z, \ell)}{(N\ell^2)^\hal} + (N\ell^2)^\hal} \right),
\end{equation*}
(see e.g. \cite[Lemma B.5]{armstrong2019local} for a precise statement), and the result follows from an application of the
``local laws'' as presented in the previous section, see e.g.~\eqref{loc_laws}. 
\end{proof}

If $\varphi$ is assumed to have more regularity, the results of \cite{LebSerCLT,bauerschmidt2019two,serfaty2020gaussian} give a much better estimate on the exponential moments of $\LN(\varphi)$, but they are only stated \emph{function-wise}. The method of this paper, which relies on those results, yields a \emph{uniform} control for the fluctuations of a certain class of $C^2$ linear statistics; see Corollary~\ref{cor:UB} and Proposition~\ref{prop:unibound} below.

\subsection{Re-writing the Laplace transform}
\corT{When computing exponential moments of fluctuations, our first step will always be the following decomposition into a simple term akin to a variance, and a certain ``ratio of partition functions'' associated to two closely related 2DOCP's.}
\begin{lemma}[Laplace transforms as ratio of partition functions]
\label{lem:rewrite_Laplace}
Let $\varphi$ be a $C^2$ function whose Laplacian is supported in $\Dd$ and  satisfies $\int \Delta\varphi=0$. Let $t, s$ be such that:
\begin{equation}
\label{condi:t}
|t| \leq \beta N |\varphi|^{-1}_{\2}, \quad s := \frac{-t}{2\pi N \beta}.
\end{equation}
Let $\mms$ be the probability measure on $\Dd$ with density $\rms := \rm_0 + s \Delta \varphi$. The following identity holds:
\begin{equation*}
\Esp \left[ e^{t \LN(\varphi)} \right] = \exp\left( \frac{t^2}{4\pi \beta} \int_{\R^2} - \varphi \Delta \varphi \right) \frac{\KNbeta(\mms)}{\KNbeta(\mm_0)}.
\end{equation*}
Of course, if $\nabla \varphi$ is compactly supported then one can integrate by parts and write the ``variance'' term $\int_{\R^2} - \varphi \Delta \varphi$ as $\int_{\R^2} |\nabla \varphi|^2$.
\end{lemma}
\begin{proof}
This follows from elementary manipulations that can be found e.g. in \cite[Section 2.6]{LebSerCLT}. It is, however, a much older idea from \cite{johansson1998fluctuations} or \cite[Sec 7.2]{ameur2011fluctuations} and the references therein.
\end{proof}

\subsection{Comparison of partition functions}
\newcommand{\tell}{\tilde{\ell}}
\newcommand{\tf}{\tilde{f}}
\corT{Note that in Lemma \ref{lem:rewrite_Laplace} the density of the ``perturbed measure'' $\mm_s$ is obtained by $\rms := \rm_0 + s \Delta \varphi$ and that clearly if $\varphi$ is compactly supported then $\int \Delta \varphi = 0$ (so in particular $\mm_s$ is again a probability measure). Most of the analysis of \cite{LebSerCLT,bauerschmidt2019two} is devoted to obtaining good asymptotics for the ratio of the partition functions associated to $\mm_s$ and $\mm$. One purpose of the next proposition is to extend this analysis to the slightly more general case where $\rms$ is obtained from $\rm_0$ by adding some perturbation $f$ such that $\int f = 0$ - without $f$ necessarily “coming” as the Laplacian of a compactly supported test function. This is of course mostly a re-writing exercise and not a major modification. Another purpose is to package together results that are written separately in the literature - in particular we incorporate Serfaty's ``smallness of the anisotropy'' trick \corO{(see \cite{serfaty2020gaussian})}
in order to have one single ``ready-to-use'' statement. }

\begin{proposition}[Main comparison result]
\label{prop:compar1}
Let $\ell$ be a lengthscale and let $z\in \Dd$ such that $(z,{\ell})$ satisfies \eqref{condiell}. Let $\Cc$ be some positive constant. Let $\mm$ be a probability measure on $\Dd$, such that its density $\rm$ is \corT{of class $C^3$, with}:
\begin{equation}
\label{assum-Main-non-bound-m}
\hal \leq \rm \leq 2, \quad  |\rm|_\kk \leq \Cc \ell^{-k} \text{ for } \kk = 1, 2, 3.
\end{equation}
Moreover, let $f$ be a function of class $C^2$ supported in $\Dd(z, \ell)$, such that:
\begin{equation}
\label{assum-Main-non-bound-f}
\int_{\R^2} f = 0, \quad |f|_\kk \leq \Cc \ell^{-(\kk+2)} \text{ for } \kk = 0, 1, 2.
\end{equation}

Then there exists a constant $\Cc'$ depending only on \corO{$\Cc$, $\beta$ and $r$} such that the following holds. For all $s \in \R$ such that:
\begin{equation}
\label{eq:condi_s}
|s| \leq \frac{1}{\Cc'} \ell^{3/2} N^{-1/4},
\end{equation}
let $\mms$ be the probability measure with density $\rms := \rm + s f$.  We have:
\begin{equation}
\label{eq:compar1}
\log \frac{\KNbeta(\mms)}{\KNbeta(\mm)}
= N \left(\frac{\beta}{4} - 1 \right) \left( \EE(\rms) - \EE(\rm) \right)  +  \corO{\Cc'}\corG{s} N^{3/4} \ell^{-1/2} \left(1 + \log\left(N \ell^2 \right) \right)^\hal.
\end{equation}
\end{proposition}
We postpone the proof of Proposition \ref{prop:compar1} to Section \ref{sec:ProofCompar1}. \corT{The proof relies almost entirely on the analysis of \cite{serfaty2020gaussian} but we could not find there a statement reasonably close to our needs.}

\begin{remark}
\label{rem:order_error}
In the statement of Proposition \ref{prop:compar1} we assume that the parameter $s$ is $O\left(\ell^{3/2} N^{-1/4}\right)$, but we eventually apply this result with $s$ of order $N^{-1}$ as in \eqref{condi:t}, which is always a valid choice for $\ell\ge \rho_\beta  N^{-1/2}$. On the other hand, the well-definiteness of $\rms$ as $\rm + s f$ only requires $s$ to be $O(\ell^{-2})$. However, for values of $s$ between $\ell^{3/2} N^{-1/4}$ and $\ell^{2}$ we do not get an interesting estimate.
\end{remark}

In Proposition \ref{prop:compar1}, we assumed
that $(z,{\ell})$ satisfies \eqref{condiell} to avoid the
case where the perturbation is near the boundary \corT{(because local laws are not known to hold near the boundary)}.
\corT{The only case where we can handle a perturbation that intersects the boundary is the one of a ``macroscopic perturbation'' ($\ell=1$), which will be enough for our purposes. This is covered by the following lemma.}

\begin{lemma}[Main comparison - macroscopic case]
\label{lem:compar_nice_bound} Let $f = 2 \pi \left( \chi - \rm_0 \right)$, where $\rm_0$ is the uniform density on $\Dd$ and $\chi$ is a smooth, radially symmetric function which is compactly supported in $\Dd_{r}$ for some \corG{fixed} $r< 1$, such that $\int \chi = 1$ and $|\chi|_{\kk} \leq \Cc$ for $\kk = 0, 1, 2$. Then there exists a constant $\Cc'$ depending only on $\Cc, \beta$ and $r$ such that the following holds: for all $s \in \R$ with $|s| \leq \frac{1}{\Cc'}$, let $\mms'$ be the probability measure with density $\rm'_s := \rm_0 + s f$.  We have:
\begin{equation}
\label{conclu-boundary}
\log \frac{\KNbeta(\mm'_s)}{\KNbeta(\mm_0)} = O(\corG{s}N),
\end{equation}
with an implicit multiplicative constant depending on $\Cc, \beta$ and $r$.
\end{lemma}
Lemma \ref{lem:compar_nice_bound} is proven along the same lines as Proposition \ref{prop:compar1}, the radial symmetry of both $f$ and $\rm_0$ provides a simplification which allows us to efficiently treat the boundary case (note also that we are aiming at a less precise estimate, compare \eqref{conclu-boundary} with \eqref{eq:compar1}). We give the proof in Section \ref{sec:ProofCompar1}.

The next result builds upon Proposition~\ref{prop:compar1} and treats a situation where the perturbation $f$ is made of two pieces living at different scales.
\begin{proposition}[Comparison with mass transfer between scales]
\label{prop:compar2}
Let $\ellA< \ellB$ be two length scales and let $\zA, \zB \in \Dd$ be two points such that both $(\zA, \ellA)$ and $(\zB, \ellB)$ satisfy \eqref{condiell}.
Let $\Cc$ be some positive constant and let $\mm$ be a probability measure on $\Dd$ such that its density $\rm$ is (\corT{on $\Dd$) of class $C^3$ with}:
\begin{equation}
\label{assum-2-Main-non-bound-m}
\frac{3}{4} \leq \rm \leq \frac{3}{2}, \quad |\rm|_\kk \leq \Cc \ellB^{-\kk} \text{ for } \kk = 1, 2, 3.
\end{equation}
Moreover, let $\fA, \fB$ be two functions of class $C^2$ supported on $\Dd(\zA, \ellA)$ (resp. $\Dd(\zB, \ellB)$), such that, with $f=\fA+\fB$,:
\begin{equation} \label{condclt}
\int_{\R^2}  f  = 0, \quad |\fA|_\kk \leq \Cc \ellA^{-(\kk+2)},\  |\fB|_\kk \leq \Cc \ellB^{-(\kk+2)} \text{ for } \kk = 0, 1, 2.
\end{equation}
{Assume that $\Dd(\zA, \ellA)$ and $\Dd(\zB, \ellB)$ are both contained in $\Dd_r$ for some $r < 1$, so that for $N$ large enough (depending on $r$), condition \eqref{condiell} is satisfied.}
Then there exists a constant $\Cc'$ depending only on $\Cc$ and $\beta$ such that the following holds. For all $s \in \R$ such that:
\begin{equation}
\label{condi:s2}
|s| \leq \frac{1}{\Cc'} \ellA^{3/2} N^{-1/4},
\end{equation}
let $\mms$ be the probability measure with density $\rms := \rm + s f$.
We have:
\begin{equation}
\label{eq:compar2}
\log \frac{\KNbeta(\mms)}{\KNbeta(\mm)}
= N \left(\frac{\beta}{4} - 1 \right) \left( \EE(\rms) - \EE(\rm) \right) +  O \left( s N^{3/4} \ellA^{-1/2} \left(1 + {\log(N \ellA^2)} \right)^\hal \right),
\end{equation}
with a multiplicative constant depending only on $\Cc, \beta$ {and $r$}.
\end{proposition}
We postpone the proof of Proposition \ref{prop:compar2} to Section \ref{sec:ProofCompar2}. \corT{It uses the same techniques as in \cite{serfaty2020gaussian}, but the statement is new.}
\begin{remark}
It is important to observe that in the conclusions of Proposition~\ref{prop:compar1} (resp. Proposition~\ref{prop:compar2}), if we work at the \emph{microscopic} scale $\ell = {\rho_\beta} N^{-1/2}$ (resp. $\ellA = {\rho_\beta} N^{-1/2}$), then for $s$ of order $N^{-1}$ (which will be our choice later on) the error term \corT{(the last term in \eqref{eq:compar1} and \eqref{eq:compar2})} is $O(1)$ \corG{as expected}, whereas as soon as the  length scale $\ell$ (resp. $\ellA$) is \emph{mesoscopic} there is a gain and the error term becomes $o(1)$.
\end{remark}

{As we record in the next remark,}
Proposition~\ref{prop:compar1} yields the classical CLT for (arbitrary smooth mesoscopic) test functions supported inside $\Dd$, see \cite{LebSerCLT} and \cite{bauerschmidt2019two} for the original results.
{A CLT-like precision will be required} to show that the exponential of a regularization of $\Pot$ converges to a certain GMC measure in Section~\ref{sec-LB}, which in turn is instrumental in obtaining a lower-bound for the maximum of $\Pot$.

\begin{remark}\label{rk:clt}
Let $\varphi \in C^4(\R^2\to\R)$ be a function (possibly depending on $N$) and assume that $\Delta \phi =f$, where $f$ {is as in Proposition \ref{prop:compar1} or Proposition \ref{prop:compar2}}.
Then by combining Lemma~\ref{lem:rewrite_Laplace} and Propositions~\ref{prop:compar1} or \ref{prop:compar2}, we get:
\begin{equation*}
\Esp \left[ e^{\LN(\varphi)} \right] = \exp\left( \frac{-1}{4\pi \beta} \int_{\R^2}  \varphi f\right) \exp\left(N \left(\tfrac{\beta}{4} - 1 \right)
\corO{(\EE(\mm)-\EE(\mm_0))} + o(1)
\right)
\end{equation*}
where $\mm =\mm_0 - \frac{f}{2\pi N \beta}$. Here we used that \corT{the uniform measure $\mm_0$ does satisfy} the conditions \eqref{assum-Main-non-bound-m}
and \eqref{assum-2-Main-non-bound-m}.
Moreover, since $f$ is supported in $\Dd$ with $\int f = 0$ and $ |f|_0 \leq \Cc \ell^{-2}$, we have
\[
\EE(\mm) \corO{-\EE(\mm_0)} = \int_{\R^2} \big(\corO{\tfrac1\pi} - \tfrac{f(x)}{2\pi N \beta}\big)\log\big(\corO{\tfrac1\pi}- \tfrac{f(x)}{2\pi N \beta}\big)
\d x-\corO{\tfrac1\pi \log\tfrac1\pi}
= o(N^{-1}).
\]
{In particular,
$\LN(\varphi)$ converges in distribution {to a Gaussian random variable} with mean 0 and variance $\frac{1}{2\pi \beta} \int_{\R^2} |\nabla \varphi|^2$ (this last expression follows by an integration by parts).}
\end{remark}

\section{Law of large numbers: upper bound}
\label{sec-UB}
The goal of this section is to prove the upper bound part of Theorem \ref{theo:main}, namely:
\begin{proposition}
\label{prop:UB}
Recall {the definition \eqref{def:PotXN} of $\Pot$}.
For all fixed $r \in (0,1)$ and all $\alpha > 1$, we have:
\begin{equation}
\label{eq:LLNUB}
\lim_{N \to \infty} \PNbeta \left( \max_{z \in \Dd(0,r)} \Pot(z) \geq \frac{\alpha \log N}{\sqrt{\beta}} \right) = 0.
\end{equation}
\end{proposition}
{In the rest of this section we fix some $r < 1$.}

\subsection{An auxiliary linear statistics}
\label{subsec-center}
For $z \in \Dd$, let $\logz : x \mapsto \log |z-x|$. The value $\Pot(z)$ of the Coulomb gas potential at $z$ corresponds to the fluctuations $\LN(\logz)$ of the linear statistics associated to $\logz$, see \eqref{def:fluctuations}.
An important technical observation is that $\int \Delta \logz = \corT{2\pi}$ so that (even after a mesoscopic regularization), one cannot directly apply Proposition~\ref{prop:compar1} to control the exponential moments of $\LN(\logz)$.
To fix this issue, we can consider instead the fluctuation of the test function $\logz-\g$ where $\g$ is some nice function with $\int \Delta \g =1$. \corT{Of course, one should be able to say something about the fluctuations $\LN(\g)$.}
It is particularly convenient to make the following choice: let $\chi$ be a radially symmetric smooth function (independent of $N$) supported in $\Dd_r$ \corT{with $\int \chi = 1$}, and let
\begin{equation}\label{eq:g}
\text{$\g$ be a solution of Poisson's equation $\Delta \g = 2\pi \chi$}.
\end{equation}
Then, the following estimate shows that the fluctuations of $\LN(\g)$ are negligible compared to the maximum of $\Pot$.

\begin{proposition}
\label{prop:expg}
One has {(for $N$ large enough depending on $\beta, r$)}
\begin{equation*}
\PNbeta \left[ |\LN(\g)| \geq (\log N)^{0.8} \right] \leq \exp\left( - \hal (\log N)^{1.5} \right).
\end{equation*}
\end{proposition}
The value $0.8$ is arbitrary, the point being that $0.8 < 1$ while $1.5>1$ and thus the probabilistic tail is better than algebraic {in $N$}.
The proof of Proposition~\ref{prop:expg} will be given in Section~\ref{sec:proofexpo_hmu} and the argument actually shows that $\LN(\g)$ is typically of order 1 as expected, see Corollary~\ref{cor:expo_hmu}.

\subsection{A regularized version of the potential}
\label{sec:regu_pot}
Let $\rho$ be a radial $C^\infty$ mollifier supported on the unit disk and for $\epsilon \in (0,1)$, let $\rhoep := \epsilon^{-2} \rho\left( \frac{\cdot}{\epsilon}\right)$ and let $\fze$ be defined as:
\begin{equation*}
\fze := \rhoep \star \logz - \g
\end{equation*}
where $\g$ satisfying \eqref{eq:g} is independent of $N,z$ and $\varepsilon$.
We think of $\fze$ as being an alternative to $\logz$ which is regularized in two ways: the singularity near $z$ is removed by convolving with a mollifier, and  the Laplacian of $\fze$ has total mass $0$.
Then, in view of Proposition~\ref{prop:expg}, in order to prove Proposition~\ref{prop:UB},
we first focus on controlling the exponential moments of $\LN(\fze)$.

\begin{proposition}
\label{prop:expo_fze}
For some constant {$\Cc\geq 1$} depending only on $\beta$, for all fixed $r \in (0,1)$, for all $N$ large enough (depending on $r$,$\beta$), for all $\epsilon$ {such that $\Cc N^{-1/2} \leq \epsilon \leq \frac{1 - r}{2}$} and for all $t$ such that $|t| \leq \Cc^{-1} N \epsilon^{2}$, for any $z\in\Dd_r$, we have:
\begin{equation}
\label{eq:expo_fze}
\Esp \left[ e^{t \LN(\fze)} \right] = \exp\left(\frac{t^2 \log {\epsilon^{-1}}}{2 \beta} + t^2 O_\epsilon(1) + t O_N(1) \right).
\end{equation}
\end{proposition}
\corT{By $O_\epsilon(1)$ we mean a term that is bounded as $\epsilon \to 0$ independently of $N$ (this term depends only on $\beta$, which is fixed here). By $O_N(1)$ we mean a term that is bounded as $N \to \infty$ independently of $\epsilon$ (this term depends on $\beta$ and $r$, which here are both fixed).}

\begin{proof}[Proof of Proposition \ref{prop:expo_fze}]
First, we impose that $\Cc$ is large enough (depending only on $\beta$) so that the length scale $\Cc N^{-1/2}$ is larger than the minimal length scale $\trhob$ {mentioned in Section \ref{sec:Energy}}. Next, observe that $\fze$ is a $C^\infty$ function whose Laplacian is supported in {$\Dd_{r + \epsilon}$} and has mean zero. Moreover, $|\fze|_k$ is of order $\epsilon^{-k}$ for any $k\in\N$. Hence, for $|t|\le \Cc^{-1} N \epsilon^{2}$ (with $\Cc$ large enough depending only on $\beta$) we may apply Lemma \ref{lem:rewrite_Laplace} and write:
\begin{equation*}
\Esp \left[ e^{t \LN(\fze)} \right] = \exp\left( \frac{t^2}{4\pi \beta} \int_{\R^2} - \fze \Delta \fze \right) \frac{\KNbeta(\mms)}{\KNbeta(\mm_0)},
\end{equation*}
with $s = \frac{-t}{2\pi N \beta}$ and $\rms := \rm_0 + s \Delta \fze$. We have:
\begin{equation*}
-  \int_{\R^2} \fze \Delta \fze = - 2\pi \int_{\R^2} \left( \log |\cdot| \star \rhoep \right) \rhoep + O_\epsilon(1),
\end{equation*}
and by scaling we can see that: $\int_{\R^2} \left( \log |\cdot| \star \rhoep \right) \rhoep =  \log \epsilon + O_\epsilon(1)$, so we may write:
\begin{equation}
\label{variance_treated}
\Esp \left[ e^{t \LN(\fze)} \right] \leq \exp\left( - \frac{t^2 \log \epsilon}{2 \beta} + t^2O_{\epsilon}(1) \right)  \frac{\KNbeta(\mms)}{\KNbeta(\mm_0)}.
\end{equation}
The assumptions of Proposition \ref{prop:compar2} are fulfilled with
$\mm= \mm_0$, $\fA = 2 \pi \rhoep$, $\fB = - 2\pi \chi$, $\zA = z$, $\zB = 0$, $\ellA = \epsilon$ and $\ellB$ fixed ({independent of $\epsilon$ and $N$}). We obtain:
\begin{equation*}
\log \frac{\KNbeta(\mms)}{\KNbeta(\mm_0)} =  N \left(\tfrac{\beta}{4} - 1 \right)\left( \EE(\rms)\corG{-\EE(\rm_0)}\right) +  O \left( s N^{3/4} \epsilon^{-1/2} \left(1 + \log (N \epsilon^2) \right)^{{\hal}}\right),
\end{equation*}
where the implicit multiplicative constant depends only on $\beta$ and $r$.
A direct computation shows that \corG{$\EE(\rms) = \EE(\rm_0) + O(s)$} -- cf.~Remark~\ref{rk:clt} -- so, as $Ns = O(t)$ and $\epsilon \geq \Cc N^{-1/2}$, we conclude that
\begin{equation*}
\log \frac{\KNbeta(\mms)}{\KNbeta(\mm_0)} =  O(t)+  O \left(t (N\epsilon^2)^{-1/4} \left(1 + \log (N \epsilon^2) \right)^\hal \right).
\end{equation*}
This completes the proof; the dominant error term being the first one in the right-hand side.
\end{proof}

\begin{corollary}
\label{coro:fze}
There exists $\Cc \geq 1$ depending on $\beta$, such that for all $\lambda \geq \Cc$, taking $\epsilon = \lambda N^{-1/2}$ and $|t| \leq \lambda^2 \Cc^{-1} $ we have:
\begin{equation}
\label{eq:corofze}
\Esp \left[ e^{t \LN(\fze)} \right] \leq \exp\left(\frac{t^2 \log N}{4 \beta} + O(t^2 + t) \right),
\end{equation}
with an implicit constant depending on $\beta, r, {\lambda}$.
\end{corollary}

\subsection{Lattice approximation and {proof of the LLN upper bound}}
\label{sec:approximation}
{We have obtained in Corollary \ref{coro:fze} a control on the exponential moments of $\LN(\fze)$ \emph{for a fixed} $z \in \Dd_r$. In the first step below, we show that it yields an upper bound \corT{on the typical values of} $\Pot(z)$  for a fixed $z$. Then we get a bound controlling the \corT{(typical)} values of $\Pot$ at \emph{all points $z$ on a sub-microscopic lattice} contained in $\Dd_r$. Finally, we turn it into a uniform control of $\Pot$ over~$\Dd_r$.}

\subsubsection*{1. Regularization and one-point tail estimate.}
First, we observe {in Lemma \ref{lem:regulPot} below} that \corT{when considering the \emph{maximal value} of $\Pot$}, we can regularize the $\log$ at the microscopic scale with a small cost. For $\epsilon > 0$, let $\PNep$ be the map:
\begin{equation}
\label{def:PNep}
\PNep : z \mapsto \LN\left( \rhoep \star \logz  \right),
\end{equation}
with $\rhoep$ as in Section \ref{sec:regu_pot}. The test function $\tlog_z := \rhoep \star \logz$ is smooth, and for $z, x \in \R^2$, we have:
\begin{equation}
\label{eq:bound_tlog}
|\nabla \tlog_z (x)| \preceq \max\left(\epsilon, |z-x|\right)^{-1},\quad  \|\corO{\mathrm{Hess}}( \tlog_z (x))\| \preceq \max\left(\epsilon, |z-x|\right)^{-2}.
\end{equation}

\begin{lemma}
\label{lem:regulPot}
For $\epsilon > 0$ and $z \in \Dd$ we have, with a universal implicit constant:
\begin{equation*}
\Pot(z) \leq \PNep(z) + O\left(N \epsilon^2 \right).
\end{equation*}
\end{lemma}
\begin{proof}
\corT{Recall that we are interested in the value of $\Pot(z)$ for some fixed $z$ and that by definition:
\begin{equation*}
\Pot(z) = \sum_{i=1}^N \logz(x_i) - N \int_{\R^2} \logz(x) \d \mm_0(x).
\end{equation*} }
By subharmonicity of $\log$, we have $\logz \leq \rhoep \star \logz$ pointwise, \corT{which implies that
\begin{equation}
\label{regulPot1}
\sum_{i=1}^N \logz(x_i) \leq \sum_{i=1}^N \rhoep \star \logz(x_i).
\end{equation}
}
On the other hand a direct computation (using e.g. Newton's theorem) yields:
\begin{equation}
\label{regulPot2}
\left| \int_{\R^2} \logz(x) \d \mm_0(x) - \int_{\R^2} \left( \rhoep \star \logz \right) (x) \d \mm_0(x) \right| = O\left(\epsilon^2\right).
\end{equation}
\corT{Combining \eqref{regulPot1} and \eqref{regulPot2} gives:
\begin{equation*}
\Pot(z) \leq \sum_{i=1}^N \rhoep \star \logz(x_i) - N \int_{\R^2} \left( \rhoep \star \logz \right) (x) \d \mm_0(x)  + N O\left(\epsilon^2\right) = \PNep(z)  + N O\left(\epsilon^2\right)
\end{equation*}
using the definition \eqref{def:PNep} of $\PNep$, which proves the claim.
}
\end{proof}
For the rest of the section, we fix:
\begin{equation}
\label{eq:choix_lambda}
\epsilon = \lambda N^{-1/2}, \quad \lambda^2 = \Cc^2 \max\left(4\sqrt{\beta},1 \right),
\end{equation}
where the constant $\Cc$ is as in Corollary \ref{coro:fze} (depending only on $\beta$). By Lemma \ref{lem:regulPot} we have, with an implicit constant depending only on $\beta$:
\begin{equation}
\label{eq:PotRegPot}
\Pot(z) \leq \PNep(z) + O\left(1 \right).
\end{equation}


\begin{lemma}
\label{lem:one_point_estimate}
Fix a point $z\in \Dd_r$ and let $\alpha\in(1, 2)$. We have:
\begin{equation}
\label{eq:one_point}
\PNbeta \left[ |\PNep(z)| \geq \frac{\alpha \log N}{\sqrt{\beta}} \right] \leq \exp\left( - \alpha^2 \log N + O(1) \right),
\end{equation}
{with an implicit constant depending on $\beta, r$}.
\end{lemma}
\corO{We note that the upper bound $\alpha<2$ is done for convenience only, one could obtain similar estimates for larger $\alpha$; for our needs, any $\alpha>1$ will do.}
\begin{proof}
Take $t : = 2 \alpha \sqrt{\beta} \leq 4 \sqrt{\beta}$ so that, in view of \eqref{eq:choix_lambda}, $t\le \lambda^2 \Cc^{-2} \leq \lambda^2 \Cc^{-1}$ with $\Cc$ as in Corollary \ref{coro:fze}. Then, we may thus use \eqref{eq:corofze} and write:
\begin{equation*}
\Esp \left[ e^{t \LN(\fze)} \right] = \exp\left(\frac{t^2 \log N}{4 \beta} + O(t^2 + t) \right) = \exp\left( \alpha^2 \log N + O(1) \right),
\end{equation*}
with an implicit constant depending only on $r, \beta$. So, Markov's inequality yields:
\begin{equation}
\label{controlfze}
\begin{aligned}
\PNbeta \left[ \LN(\fze) \geq \frac{\alpha \log N}{\sqrt{\beta}}  \right]
&\leq \exp\left( - \frac{t \alpha \log N}{\sqrt{\beta}}  \right)\Esp \left[ e^{t \LN(\fze)} \right]\\
&\leq \exp\left( - \alpha^2 \log N + O(1) \right).
\end{aligned}
\end{equation}
Note that since \eqref{coro:fze} is valid for $t\in\R$, the same bound holds for $- \LN(\fze)$ as well.
On the other hand,  by Proposition~\ref{prop:expg}, we have $|\LN(\g)| \leq (\log N)^{0.8}$ with probability $1 - \exp\left( - (\log N)^{1.5}\right)$. Since by definition $\LN(\fze) = \PNep(z) + \LN(\g)$, we may convert the control \eqref{controlfze} on $\LN(\fze)$ into a similar estimate for {$\PNep(z)$, as stated.}
\end{proof}

\subsubsection*{2. Tail estimate on a lattice.}
For $\alpha\in(1, 2)$ and fix $\delta$ (depending on $\alpha$) such that:
\begin{equation}
\label{choixdelta}
0 < \delta < \frac{\alpha^2 -1}{2}
\end{equation}
Let $\elld := N^{-\hal - \delta}$ (in particular, $\elld$ is sub-microscopic) and consider the lattice $\Latt:=\left(\elld \Z\right)^2$.
\begin{lemma}[Tail estimate on the lattice]
\label{lat-est}
Let $\alpha$ be in $(1, 2)$. We have:
\begin{equation}
\label{eq:tail_on_lattice}
\PNbeta \left[ \max_{z \in \Latt \cap \Dd(0,r)} |\PNep(z)| \geq \frac{\alpha \log N}{\sqrt{\beta}}  \right]  \leq \exp\left( (1 + 2\delta) \log N - \alpha^2 \log N + O(1) \right),
\end{equation}
with an implicit constant depending on $\beta, r$. {In particular we have:
\begin{equation}
\label{one-point-decay}
\PNbeta \left[ \max_{z \in \Latt \cap \Dd(0,r)} |\PNep(z)| \geq \frac{\alpha \log N}{\sqrt{\beta}}  \right] \to 0,
\end{equation}
as $N \to \infty$  -- the decay being in fact algebraic in $N$.}
\end{lemma}
\begin{proof}
The number of points of $\Latt$ that fall into $\Dd_r$ is bounded above and below by positive constants (depending only on $r$) times $N^{1 + 2 \delta}$. Thus \eqref{eq:tail_on_lattice} follows from a simple union bound, using the one-point estimate \eqref{eq:one_point}. {From our choice \eqref{choixdelta} for $\delta$ we easily deduce \eqref{one-point-decay}.}
\end{proof}

\subsubsection*{{3. Extension to the whole disk}}
It remain to prove a result comparing the maximum of the Coulomb gas potential over the lattice points versus the whole unit disk.
Recall that $\varepsilon$ corresponds to some ``large enough microscopic'' scale, see \eqref{eq:choix_lambda}.
\begin{proposition}[The lattice vs. the whole disk]
\label{prop:lattice}
One has
\begin{equation}
\label{eq:lattice_vs_disk}
\PNbeta \left[ \max_{z \in \Dd(0,r)} \PNep(z) \geq 1 + \max_{z' \in \Latt \cap \Dd(0,r)} \PNep(z') \right] \to 0 \text{ as } N \to \infty.
\end{equation}
\end{proposition}

\begin{proof}[Proof of Proposition \ref{prop:lattice}]
\corT{Let $z\in \Ddr$, and let $z'$ be a point in $\Latt \cap \Ddr$ such that $|z-z'| \leq 4 \elld$ (such a point always exists because the lattice has been chosen narrow enough)}. We want to show that, with high probability, the values of the regularized potential at $z$ and $z'$ are close.

Recall the definition \eqref{def:PNep} of $\PNep$. We can write the difference $\PNep(z) -  \PNep(z')$ as $\LN( \Gzz)$, where $\Gzz$ is the function
\begin{equation*}
\Gzz(x) := \rhoep \star \logz(x) - \rhoep \star \logzp(x) = \tlog_x(z) - \tlog_x(z').
\end{equation*}
In view of the bounds \eqref{eq:bound_tlog} on the derivatives of $\tlog$, and since $|z-z'| \leq 4 \elld$, we get:
\begin{equation}
\label{estimeeGzz2}
\begin{aligned}
&|\Gzz(x)| \preceq  \elld \epsilon^{-1},&|\nabla \Gzz(x)| \preceq \elld \epsilon^{-2} \text{ for } x \in \Dd(z', 4 \epsilon),\\
&|\Gzz(x)| \preceq  \elld |z-x|^{-1},& |\nabla \Gzz(x)| \preceq \elld |z-x|^{-2}  \text{ for } x \notin \Dd(z', 4 \epsilon).
\end{aligned}
\end{equation}
\corT{In order to control the difference between the values of $\Pot$ at the lattice points and on the whole disk, we would like to control the size of \corG{$\LN(\Gzz)$} \emph{uniformly} for all $(z,z')$ as above. Lemma \ref{lem_fluct_uniform} does provide such a uniform control, but it applies to functions with \emph{a given Lipschitz constant}, which is not the case of $\Gzz$, which is intrinsically multi-scale - see \eqref{estimeeGzz2}. We thus need to extend this result from a ``single-scale'' statement to a ``multi-scale'' one, the same way that Proposition \ref{prop:compar2} extends Proposition \ref{prop:compar1}.}

\vspace{0.2cm}

\corT{\textbf{1. Dyadic scale decomposition.}}
We introduce a sequence of intermediate length scales $\ell_0 < \dots < \ell_n$ by taking $c, n$ such that:
\begin{equation*}
n = \floor{\log \left( \frac{1 -r}{4 \epsilon} \right)}, \quad \log c := \frac{1}{n} \log \left( \frac{1 -r}{4 \epsilon} \right),
\end{equation*}
and by setting $\ell_k := c^k \epsilon$ for $k = 0, \dots, n$. Note that
$c\in [1,2]$, that $n = O(\log N)$, that $\ell_0 = \epsilon$ and that $\ell_n = \frac{1 -r}{4}$. We also set $\ell_{n+1} = M$, with $M$ large (depending on $\beta$ but not on $N$) to be chosen later. Next, we take a family $(\chi_i)_{0 \leq i \leq n+1}$ of functions such that:
\begin{equation*}
\sum_{i=0}^{n+1} \chi_i \equiv 1 \text{ on } \Dd(0, M), \quad \sum_{i=0}^{n+1} \chi_i \equiv 0 \text{ outside } \Dd(0, 2M),
\end{equation*}
each $\chi_i$ living at scale $\ell_i$ around $z'$, namely:
\begin{equation*}
|\chi_i|_\0 \leq 1, \quad |\chi_i|_\1 \preceq \ell_i^{-1}, \quad \chi_i \equiv 0 \text{ outside } \Dd(z', 4 \ell_i)\setminus \Dd(z',\ell_1/2),
\end{equation*}
and we let $\Gzz_i := \chi_i \Gzz$ for $0 \leq i \leq n$.

\vspace{0.2cm}

\corT{\textbf{2. Controlling every scale}}
Using the bounds on $\Gzz, \nabla \Gzz$ and the ones on $\chi_i, \nabla \chi_i$, a direct computation ensures that $\Gzz_i$ is (up to some multiplicative constant) $\elld \ell_i^{-2}$-Lipschitz for each $i$ - in other words, the function $\tG_i := \frac{\ell_i}{\elld} \Gzz_i$ is $\frac{1}{\ell_i}$-Lipschitz and thus satisfies the assumptions of Lemma \ref{lem_fluct_uniform}. We may write, using a convexity inequality for $\exp$:
\begin{multline*}
\exp\left( N^{\delta/2} \sum_{i=0}^{n+1} \LN(\Gzz_i) \right) = \exp\left(N^{\delta/2} \sum_{i=0}^{n+1} \frac{\elld}{\ell_i} \LN(\tG_i) \right) \\
\preceq \frac{1}{\log N} \sum_{i=0}^{n+1} \exp\left( \frac{N^{\delta/2} \elld \log N}{\ell_i} \LN(\tG_i) \right),
\end{multline*}
where we used that $n$ is of order $\log N$. Since $\frac{N^{\delta/2} \elld \log N }{\ell_i} \ll 1$, we may use the control in exponential moments given by Lemma \ref{lem_fluct_uniform} and write:
\begin{equation*}
\log \Esp \left[  \exp\left(\frac{N^{\delta/2} \elld \log N }{\ell_i} \LN(\tG_i) \right) \right] = O\left(\frac{N^{\delta/2} \elld \log N }{\ell_i} \left(N \ell_i^2\right)^\hal   \right) = O\left(N^{-\delta/2} \log N \right).
\end{equation*}
We obtain a similar control for the (macroscopic) scale $\ell_{n+1}$, see the last comment in Lemma~\ref{lem_fluct_uniform}. Then Markov's inequality yields:
\begin{equation}
\label{eq:sumGismall}
\PNbeta \left[  \sum_{i=0}^{n+1} \LN(\Gzz_i)  \geq 1 \right] \leq \exp\left(- N^{\delta/2} \right).
\end{equation}
We emphasize that since we are using Lemma \ref{lem_fluct_uniform}, which provides uniform control for fluctuations of all Lipschitz functions living at scale $\ell$ around a given point, then for fixed $z'$, the control \eqref{eq:sumGismall} is in fact uniform in $z$ i.e.
\begin{equation}
\label{comparaison_reseau}
\PNbeta \left[ \max_{z, |z-z'| \leq 4 \elld}  \sum_{i=0}^{n+1} \LN(\Gzz_i) \geq 1 \right] \leq \exp\left(- N^{\delta/2} \right).
\end{equation}

\vspace{0.2cm}

\corT{\textbf{3. Handling the outliers.}}
On the other hand, since by construction $\sum_{i=0}^{n+1} \chi_i \equiv 1 \text{ on } \Dd(0, M)$, any difference between $\PNep(z) -  \PNep(z') = \LN(\Gzz)$ and $\sum_{i=0}^{n+1} \LN(G_i)$ comes from hypothetical outliers living far away from the unit disk, namely outside $\Dd(0, M)$. A simple large deviation
estimate (see e.g. \cite[(1.48)]{sandier20152d}) shows that the probability of any point being outside $\Dd(0, M)$ decays as $\exp\left(-\beta N M^2\right)$ as $N \to \infty$ for $M$ large enough depending only on $\beta$ (indeed the non-negative confinement term $2 \zeta(x)$ appearing in the Boltzmann factor grows as $|x|^2$ for large $x$ as can be seen in \eqref{def:zeta}).

\vspace{0.2cm}

\corT{\textbf{4. Conclusion of the proof of Proposition \ref{prop:lattice}.}}
Finally, we obtain (by symmetry) for any fixed $z'$:
\begin{equation}
\label{comparaison_reseau_2}
\PNbeta \left[ \max_{z, |z-z'| \leq 4 \elld} \left|  \PNep(z) -  \PNep(z') \right| \geq 1 \right] \leq  \exp\left(- N^{\delta/2} \right).
\end{equation}
A union bound over $z' \in \Latt \cap \Dd$ concludes the proof.
\end{proof}

{Combining Lemma~\ref{lat-est} and Lemma \ref{prop:lattice} proves the LLN upper bound stated in Proposition~\ref{prop:UB}. Moreover, using \eqref{comparaison_reseau_2}, we obtain a quantitative tail estimate}:
\begin{corollary}\label{cor:UB}
For all $\alpha > 1$, we have
\begin{equation*}
\PNbeta \left[ \max_{z \in \Dd(0,r)} \left| \PNep(z) \right| \geq \frac{\alpha \log N}{\sqrt{\beta}}  \right]  \leq \corT{C(\alpha) \exp\left(- \frac{(\alpha -1)}{4} \log N \right)}.
\end{equation*}
\end{corollary}
\begin{proof}
Combining \eqref{eq:tail_on_lattice} and \eqref{comparaison_reseau_2} and choosing $\delta = \alpha -1$ we obtain:
\begin{equation*}
\PNbeta \left[ \max_{z \in \Dd(0,r)} \left| \PNep(z) \right| \geq \frac{\alpha \log N}{\sqrt{\beta}}  \right] \leq \exp\left( - \frac{(\alpha -1)}{2} \log N + O(1) \right) +  \exp\left(- N^{(\alpha-1)/4}\right),
\end{equation*}
which yields the claim.
\end{proof}

\subsection{Uniform control of fluctuations for smooth linear statistics}
As a byproduct, we obtain uniform bounds for the fluctuations of linear statistics for a class of $C^2$ test functions with, say, $|\cdot|_2 \le 1$. \corT{This can be thought of as an analogue of Lemma \ref{lem_fluct_uniform} for test functions that are smoother. The basic idea is to use integration by parts to translate the question of bounding $\LN[f]$ into a bound on $\Pot$ times $|f|_\2$.} \corT{Recall that $r < 1$ is fixed and let}
\begin{equation*}
\mathcal{F}_N : = \left\lbrace f \in C^2({\mathrm{D}_r}), \text{ $f$ {is supported in some disk} $\Dd(z, \ell)$ \text{ with } $\ell \ge \lambda N^{-1/2}$, and $|f|_2 \le \ell^{-2}$}\right\rbrace,
\end{equation*}
where $\lambda$ is the constant depending only on $\beta$ chosen above in~\eqref{eq:choix_lambda}.
\begin{proposition}\label{prop:unibound}
For any $k\in\N$, there is a constant $C_k = C_k(\beta)$ so that if $N$ is sufficiently large (depending on $\beta$, $k$ and $r$),
\begin{equation*}
\PNbeta \left[ \sup_{f\in\mathcal{F}_N} \left| \LN(f) \right| \geq C_k \log N  \right]  \leq N^{-k}
\end{equation*}
\end{proposition}
\begin{proof}
We use again $\epsilon = \lambda / \sqrt{N}$. Let $(\square_j = \square(z_j,\epsilon))_{j=1}^M$ be a collection of squares centered at points $z_j\in \epsilon \Z^2$ such that $\Dd(0,r) \subset \bigcup_{j=1}^M \square_j \subset \Dd(0,1-\delta)$ for a small $\delta>0$. {The collection can be chosen so that}
$M \leq \Cc_\beta N$ for some constant depending on $\beta$. Using the local laws \eqref{nbest} and a union bound, we deduce that for all $k\in\N$, there is a constant $C_k = C_k(\beta)$ such that if $N$ is sufficiently large,
\begin{equation}
\label{eventsquare}
\PNbeta \left[ \max_{j\le M} \bXN(\square_j) \le C_k \log N \right] \ge 1- N^{-k}.
\end{equation}
Let $f\in \mathcal{F}_N$, $\ell$ be the associated length scale, and let $f_\epsilon = f \star \rho_\epsilon$.
Since $\rho$ is a radial mollifier (with compact support in $\Dd$) and $|f|_2\le \ell^{-2}$, one has
\[
|f_\epsilon - f |_0  \le  (\epsilon/\ell)^2
\]
and both $f,f_\epsilon$ are compactly supported in $ \Dd(z, 2\ell)$  since $\ell\ge \epsilon$. One can tile $\Dd(z, 2\ell)$ with at most $16(\ell/\epsilon)^2$ squares of sidelength $\epsilon$, thus on the event introduced in \eqref{eventsquare} we have:
\begin{equation} \label{fapprox}
| \LN(f_\epsilon)-  \LN(f)| \le 16 C_k (1+\log N) .
\end{equation}
Moreover, by an integration by parts we can write $\int \Delta f \,  \PNep = \LN(f_\epsilon)$, and using again that $|f|_2\le \ell^{-2}$ and that  $\Delta f$ is supported in $\Dd(0,r)$, we deduce:
\[
|\LN(f_\epsilon)| \le \max_{z \in \Dd(0,r)} \left| \PNep(z) \right| .
\]
The RHS is controlled by Corollary~\ref{cor:UB} and, by \eqref{fapprox}, on the event introduced in \eqref{eventsquare} we can use this to control $\LN(f)$ \emph{uniformly} for all $f\in \mathcal{F}_N$.
Adjusting the constants $C_k$, this proves the claim.
\end{proof}

\section{Regularized multiplicative chaos: lower bound}
\label{sec-LB}
This section is devoted to the proof of the lower bound in Theorem \ref{theo:main}.

\subsection{{Reduction to a regularized version of the potential}}
Recall that  $\fze := \rhoep \star \logz - \g$, see \eqref{eq:g}, is a regularization of $\logz$ at scale $\epsilon>0$ and let $\mathcal{U} \subset \Dd$ be a (non-empty) open ball.
By Proposition~\ref{prop:expg}, we know that for any (small) $\delta>0$:
\begin{equation*}
\liminf_{N\to\infty}\P\left[ \sup_{z\in\mathcal{U}} \LN(\rhoep \star \logz)  \ge (\beta^{-1/2}-\delta) \log N \right] \ge \liminf_{N\to\infty} \P\left[ \sup_{z\in\mathcal{U}} \LN(\fze)  \ge (\beta^{-1/2}-\delta/2) \log N \right].
\end{equation*}
In addition, recall that $\Pot(z) = \LN(\logz)$ so that
\[
\LN(\rhoep \star \logz) =  \LN\bigg( \int \rhoep(x) \log_{z-x}(\cdot) {\d x} \bigg)
= \int \rhoep(x) \Pot(z-x) {\d x}.
\]
The right-hand side is the convolution of the function $\Pot$ {with a non-negative function of total mass $1$} and as such \corO{(since the maximum of the 
convolution
is less than or equal to the maximum of $\Pot$),}
\[
\max_{z\in \mathcal{K}}  \Pot(z) \ge \sup_{z\in \mathcal{U}}   \LN({\rhoep \star \logz}),
\]
where $\mathcal{K}$ is a closed ${\epsilon}$-neighborhood of $\mathcal{U}$.

Altogether, this implies that for any (small) $\delta>0$,
\begin{multline}
\label{LB1}
\liminf_{N\to\infty}\P\left[ \max_{z\in \mathcal{K}}  \Pot(z) \ge (\beta^{-1/2}-\delta) \log N \right] \\
\ge \liminf_{N\to\infty} \P\left[ \sup_{z\in\mathcal{U}}  \LN(\fze)  \ge (\beta^{-1/2}-\delta/2) \log N \right].
\end{multline}

Hence, it suffices to show that the probability on the RHS of \eqref{LB1} converges to 1 as $N\to\infty$. It turns out that this is the case if the scale $\epsilon(N) \le N^{\nu-1/2} $ for some $\nu>0$
sufficiently small (depending on $\delta>0$).
For technical reasons, we also need that $\epsilon(N) \gg N^{-1/2} $ is larger than the microscopic scale. We claim the following:

\begin{proposition} \label{prop:maxLB}
Let $\epsilon(N)$ be a sequence so that $\epsilon(N) \to 0$ as $N\to\infty$ in such a way that 
$\epsilon(N)\gg N^{-1/2} $.
Then, for any $\corO{\alpha<2}$,
\begin{equation}
\label{maxLB1}
\lim_{N\to\infty} \P\left[ \sqrt{\beta} \sup_{z\in\mathcal{U}}  \LN(\fze)  \ge \alpha \log \epsilon(N) ^{-1} \right]  = 1 .
\end{equation}
\end{proposition}
Combining \eqref{LB1} and \eqref{maxLB1}, we get: $\lim_{N\to\infty}\P\left[ \sup_{z\in \mathcal{U}}  \Pot(z) \ge \bar\alpha \log N \right] =1$ for any value $\corO{\bar\alpha < \beta^{-1/2}}$, which yields the lower bound in Theorem \ref{theo:main}.

The proof of Proposition~\ref{prop:maxLB} relies on the theory of Gaussian multiplicative chaos and in particular on {\cite{CFLW21} and \cite{LOS18}.}
We review these results in the next section and present the main steps of the proof.

\subsection{Multiplicative chaos.}
We  introduce new notations.
Let $\ell(k)= e^{-k}$ for $k\in\N$ and set
\[
\Phi_{k}(z) := \sqrt{\beta}\, \LN(\varphi_{z,\ell(k)})  , \qquad z\in \Dd.
\]
For $\gamma>0$, define a sequence of random measure $(\boldsymbol\mu_k^\gamma)_{k\in\N}$ with density function
\begin{equation}
\label{eq-mugamma}
\mu_{k}^\gamma = \frac{e^{\gamma\Phi_k}}{\Esp e^{\gamma \Phi_k}} .
\end{equation}
This density obviously depends on the $N$ and $\beta$ even though it is not emphasized {in the notation.}

\medskip

According to Proposition~\ref{prop:expo_fze}, if  $\ell(k)\geq \Cc N^{-1/2}$, it holds uniformly for $x\in\mathcal{K}$,
\begin{equation} \label{bd1}
\Esp[ e^{\gamma \Phi_k(x)} ] = \exp\left(\frac{\gamma^2 k}{2}  +  O_N(1) \right)
\asymp \ell(k)^{-\gamma^2/2} .
\end{equation}
These asymptotics corresponds to {\cite[Assumptions 3.1]{CFLW21}.}
In fact, using the method developed in  \cite[Section 3]{CFLW21} and \cite[Section 2]{LOS18},
we will obtain the following  convergence result (with respect to the vague topology for positive measures on $\Dd$).

\begin{proposition} \label{prop:MC}
Let $n(N)$ be a sequence such that $n(N)\to\infty$ and ${\ell(n):=\ell(n(N))} \gg N^{-1/2}$ as $N\to\infty$.
Then for any $\corO{0<\gamma<2}$, $\boldsymbol\mu_{n(N)}^\gamma \to  \operatorname{GMC}_\gamma$ in distribution as $N\to\infty$ where $\operatorname{GMC}_\gamma$ is a Gaussian multiplicative chaos measure which is defined shortly. \end{proposition}

\begin{remark} \label{rk:MC}
If one consider the sequence of random measures with densities
\(
\hat\mu_{N}^\gamma = \frac{e^{\gamma\sqrt{\beta} \Pot }}{\Esp e^{\gamma\sqrt{\beta} \Pot}}
\)
for $N\in\N$ instead of \eqref{eq-mugamma}, we expect that the result of Proposition~\ref{prop:MC} remains true, that is, for any $0<\gamma<2$,  $\boldsymbol{\hat\mu}_{N}^\gamma \to  \operatorname{GMC}_\gamma$ in distribution as $N\to\infty$ for a slightly different GMC (associated to the GFF with \emph{free boundary condition} on $\Dd$).
Note that the regime $0<\gamma<2$ corresponds to the whole GMC $L^1$-phase, as $\gamma=2$ is the critical value with our normalization.
\end{remark}

We now turn to the definitions of the random measures $(\operatorname{GMC}_\gamma)_{0<\gamma<2}$.
We consider the following log-correlated field.
\corO{Let $\mathcal{C}_c(\Dd)$ denote the space of smooth, compactly supported 
functions on $\Dd$.}

\begin{definition} \label{def:GMC}
Let  $\Psi$ be a (distribution-valued) Gaussian process on $\D$ with mean zero and the following correlation structure: for any $f,\corO{h} \in \mathcal{C}_c(\Dd)$ with $\int f = \int \corO{h} =1$,
\begin{equation} \label{GFF}
\Esp\big[\Psi(f) \corO{\Psi(h)}\big] 
:=  \frac1{(2\pi)^2} \iint \big(f(x)-\chi(x)\big)\big(\corO{h(z)}-\chi(z)\big) \log|x-z|^{-1} \d x \d z
\end{equation}
\corO{with $\chi$ a radially symmetric smooth function (independent of $N$) supported in $\Dd_r$ \corT{with $\int \chi = 1$}, as used in~\eqref{eq:g}.}
It is well-known that the RHS defines the covariance of a Gaussian process
(this also follows from the CLT of Remark~\ref{rk:clt} which implies that with $\varphi =- \Delta^{-1}(f-\chi) : z\mapsto 2\pi\int \log|x-z|^{-1} (f(x)-\chi(x)) \d x$, one has $\LN(\varphi) \to \Psi(f)/\sqrt\beta$ in distribution),
so that \corO{with $\g$ as in \eqref{eq:g},}
the field $\Psi$ has correlation kernel
\[
\Sigma(x,z) =  \log|x-z|^{-1} + \g(x) + \g(z) + c,
\]
where $c\in \R$ is a constant, \corO{that is the right hand side of \eqref{GFF} can be written as 
$$\iint f(x) h(z)\Sigma(x,z) \d z \d x.$$}

For any $k\in\N$, we define $\Psi_k :=  \rho_{\ell(k)} \star \Psi $. This is a (smooth) approximation of $\Psi$ as $k\to\infty$ and, for $\gamma>0$, we also let
$\boldsymbol\nu_{k}^\gamma$ be a random measure on $\D$ with density function
\[
\nu_{k}^\gamma(x) = \frac{e^{\gamma\Psi_k(x)}}{\Esp e^{\gamma \Psi_k(x)}}  ,  \qquad x\in \Dd .
\]
\end{definition}

Then, the following convergence result follows from the general theory of multiplicative chaos; e.g.~\cite{Berestycki17}.

\begin{proposition} \label{prop:GMC}
For any $\gamma <2$ (subcritical phase), the random measure $\boldsymbol\nu_{k}^\gamma \to \operatorname{GMC}_\gamma$ in probability as $k\to\infty$.
Moreover, for any (non-empty) open set $A \subset \Dd$,
$\operatorname{GMC}_\gamma(A)>0$ almost surely.
\end{proposition}

Hence, in the \emph{subcritical phase} ($\gamma<2$), as a consequence of Proposition~\ref{prop:MC} and \cite[Theorem 3.4]{CFLW21}, we
obtain Proposition~\ref{prop:maxLB}.
Finally, the proof of  Proposition~\ref{prop:MC} will be a direct application of {\cite[Theorem 2.4]{CFLW21} (reproduced there from \cite[Theorem 1.7]{LOS18}).}
Namely, it suffices to show that for any $j\in\N$, $k_1, \cdots , k_j \le n(N)$ and $\gamma_1,\cdots, \gamma_j \in
\R$,
\begin{equation} \label{asymp_mom}
\Esp\big[ e^{\gamma_1 \Phi_{k_1}(x_1) + \cdots + \gamma_j \Phi_{k_j}(x_j) } \big] =
\Esp\big[ e^{\gamma_1 \Psi_{k_1}(x_1) + \cdots + \gamma_j \Psi_{k_j}(x_j) } \big] \big(1+{o(1)} \big)
\end{equation}
uniformly for $x_1, \cdots, x_j\in \Dd_r$ (here $r<1$ is fixed).
The error term is also uniform for $k_1, \cdots k_j \le n(N)$, in which case the condition
$\ell(n) \gg N^{-1/2}$ as $N\to\infty$ is crucial.
The next section is devoted to the proof of \eqref{asymp_mom}.

\subsection{Exponential moments asymptotics.}
We deduce \eqref{asymp_mom} from the estimates of Proposition~\ref{prop:compar2} by a simple induction.

First observe that since $\Phi_{k}(x) = \sqrt{\beta}\, \LN(\varphi_{x,\ell(k)})$ with $\Delta \varphi_{x,\ell}=  \rho_{x,\ell}- \g$,  as a consequence of Remark~\ref{rk:clt} (with $f_a = \rho_{x,\ell}=\rho((\cdot-x)\ell^{-1})\ell^{-2}$ and  $f_b = \g$ so that the conditions \eqref{condclt} hold), we have
\[
\Esp\big[ e^{\gamma \Phi_{k}(x)}\big]
=  \exp\left( \frac{-\gamma^2}{4\pi} \int_{\R^2}  \varphi_{x,\ell(k)} \Delta \varphi_{x,\ell(k)} \right) \big(1+o(1) \big)
\]
uniformly for $x\in \Dd_r$, $\ell(k) \gg N^{-1/2}$ and locally uniformly for $\gamma\in\R$.
Moreover (by definition of the 2d Green's function for $-\Delta$ and by Fubini's theorem), according to Definition~\ref{def:GMC}, we have for $k,n \in\N$  and $x,z \in\Dd$,
\begin{align}\notag
\frac{-1}{2\pi} \int  \varphi_{x,\ell(k)} \Delta \varphi_{z,\ell(n)}
& = \iint  \big(\rho_{x,\ell(k)}(u)-\g(u)\big)\big(\rho_{z,\ell(n)}(v)-\g(v)\big) \log|u-v|^{-1} \d u \d v  \\
\label{cov}& = \Esp[\Psi_k(x) \Psi_n(z)] .
\end{align}
Thus, since $\Psi$ is a (mean-zero) Gaussian process, we obtain for any $x\in \Dd$, $\gamma \in\R$ and  \corG{$\ell(k) \gg N^{-1/2}$}, 
\[
\Esp\big[ e^{\gamma \Phi_{k}(x)}\big] = \Esp\big[ e^{\gamma \Psi_{k}(x)}\big] \big(1+
{o(1)} \big)
\]
with the required uniformity. This establishes that \eqref{asymp_mom} holds when $j=1$.

We now proceed by induction to extend these asymptotics for any $j\in \N$ with $j\ge 2$.
Without loss of generality, we assume that $k_1\le \cdots \le  k_j \le n(N)$. Then, according to Lemma~\ref{lem:rewrite_Laplace}, we have
\[\begin{aligned}
\Esp\big[ e^{\gamma_1 \Phi_{k_1}(x_1) + \cdots + \gamma_j \Phi_{k_j}(x_j) } \big]
&= \Esp \big[ e^{\sqrt{\beta} \LN(\varphi_j)} \big]  \\
&= \exp\left( \frac{-1}{4\pi} \int_{\R^2} \varphi_j \Delta \varphi_j \right) \frac{\KNbeta(\mm_0 +  s \Delta \varphi_j)}{\KNbeta(\mm_0)}
\end{aligned}\]
where $\varphi_j =\sum_{i\le j}  \gamma_i \varphi_{x_i,\ell(k_i)} $ and $s= \frac{-1/\beta^{1/2}}{2\pi N}$.
Moreover, using \eqref{cov}, we obtain
\[
\Esp\Big[ e^{ \sum_{i\le j}\gamma_i \Phi_{k_i}(x_i)} \Big]
=\Esp\big[ e^{\gamma_1 \Psi_{k_1}(x_1) + \cdots + \gamma_j \Psi_{k_j}(x_j) } \big]  \frac{\KNbeta(\mm_0 +  s \Delta \varphi_j)}{\KNbeta(\mm_0)}
\]
and, by taking a ratio,
\[
\frac{\Esp\Big[ e^{ \sum_{i\le j}\gamma_i \Phi_{k_i}(x_i)} \Big]}{\Esp\Big[ e^{ \sum_{i\le j}\gamma_i \Psi_{k_i}(x_i)} \Big]}
= \frac{ \Esp\Big[ e^{ \sum_{i< j}\gamma_i \Phi_{k_i}(x_i)} \Big]} { \Esp\Big[ e^{ \sum_{i< j}\gamma_i \Psi_{k_i}(x_i)} \Big]}   \frac{\KNbeta(\mm_0 +  s \Delta \varphi_j)}{\KNbeta(\mm_0+  s \Delta \varphi_{j-1})} .
\]
We now apply Proposition~\ref{prop:compar2} with $f_a =  \gamma_j \rho_{x_j,\ell(k_j)}$ and
$f_b= \gamma_j\g$ as above.
We emphasize that the conditions \eqref{condclt} hold while the reference measure $\mm = \mm_0 +  s \Delta \varphi_{j-1}$ also satisfies $|\rm - 1| \le \Cc \ell(n)^2/N$.
In particular, in the regime $\ell(n) \ll \sqrt{N}$, we obtain
\[
\log\bigg( \frac{\KNbeta(\mm_0 +  s \Delta \varphi_j)}{\KNbeta(\mm_0+  s \Delta \varphi_{j-1})} \bigg)
= N \left(\frac{\beta}{4} - 1 \right) \left( \EE(\rm+  s \Delta \varphi_j) - \EE(\rm+  s \Delta \varphi_{j-1}) \right) + {o(1),}
\]
with the required uniformity.
Now, repeating the argument from Remark~\ref{rk:clt}, the entropy
$ \EE(\rm+  s \Delta \varphi_j)  = o(N^{-1})$ for any $j\in\N$ which implies that
\[
\frac{\KNbeta(\mm_0 +  s \Delta \varphi_j)}{\KNbeta(\mm_0+  s \Delta \varphi_{j-1})} = 1+{o(1)}  .
\]
By induction, this concludes the proof of the asymptotics \eqref{asymp_mom}.
\qed

\appendix
\section{Auxiliary proofs}
\subsection{Proof of Proposition \ref{prop:compar1} and Lemma \ref{lem:compar_nice_bound}}
\label{sec:ProofCompar1}
\corT{Before going into the proof, which requires a careful technical reading of \cite{serfaty2020gaussian}, let us present the main idea. Proposition \ref{prop:compar1} is a \emph{comparison} statement: one wants to compare two partition functions for the 2DOCP's associated to two closely related measures: the ``reference'' measure $\mm$ and its perturbation $\mm_s$ whose density $\rm_s$ is a perturbation of $\rm$ of the form $\rm_s = \rm + s f$. \\
In order to ``map'' one 2DOCP onto the other, one starts by mapping one measure onto the other. There are many ways to do that exactly i.e. to construct ``transportation'' maps that will exactly push-forward $\mm$ onto $\mm_s$, yet for technical reasons we prefer to introduce an ``approximate transport'' map $\Phi_s$ which sends $\mm$ to something close to $\mm_s$ (the difference is of second order in $s$ when one considers the densities) but has the upside of admitting an expression as $\Phi_s = \id + s \psi$ with $\psi$ independent of $s$. This can be thought of as looking at a linearization of exact transportation maps to first order in $s$. One needs to estimate the error induced by this approximation, which is done in \cite{serfaty2020gaussian}.\\
One is then left with another comparison between two partition functions associated to $\mm$ and $\tmms$ (the ``approximate'' perturbed measure obtained with the ``approximate'' transport map $\Phi_s$). It is natural to extend $\Phi_s : \R^2 \to \R^2$ into a change of variables on the whole phase space $(\R^2)^N$ by acting coordinate-wise. Applying this change of variables in the integral defining the partition function, one is led to consider:
\begin{enumerate}
\item The Jacobian of the diffeomorphism.
\item The effect of $\Phi_s$ on the energy.
\end{enumerate}
The first point is fairly easy to deal with. The second point is much more subtle and was analyzed in \cite{LebSerCLT,bauerschmidt2019two,serfaty2020gaussian}, the result being an expansion (in $s$) of the energy along a transport of the form $\Phi_s = \id + s \psi$ with good controls on each term of the expansion. The first order in $s$ corresponds to a term that was called the ``anisotropy'' in \cite{LebSerCLT,serfaty2020gaussian} or the ``angle term'' in \cite{bauerschmidt2019two} and is absolutely central for the analysis. We import the relevant results in their up-to-date formulation of \cite{serfaty2020gaussian}.\\
An additional complication is the following fact: a purely ``analytic'' study of the anisotropy term yields a certain bound on it which is enough for many purposes, however, for an optimal understanding of the ratio of partition functions (which is needed here) this deterministic bound fails to be good enough. 
In both \cite{LebSerCLT} and \cite{bauerschmidt2019two}, one resorts to a probabilistic trick which can be summarized as follows: on the one hand, the anisotropy term appears in the expansion of the energy and thus in the comparison of partition functions as a function of $s$. On the other hand, by completely other means, it is possible to get an expression for the \emph{same ratio of partition functions} (when $s$ is large). To leading order, this alternative expression does \emph{not contain anything like the anisotropy term}, which has no choice but to be contained in the error term. One can thus conclude that the anisotropy is ``smaller than it seems'' \corG{(in the sense of exponential moments)} when $s$ is large, but since this quantity is linear in $s$ it is not hard to get the same conclusion for $s$ of the size that is relevant to us.
}

\subsubsection{Proof of Proposition \ref{prop:compar1}}
We combine the following tools from \cite{LebSerCLT,serfaty2020gaussian}.
\begin{enumerate}
\item Approximate transport and estimates on the approximation error.
\item Comparison of partition functions along  the approximate transport \& the anisotropy term.
\item Relative expansion of partition functions.
\item ``Serfaty's trick'' for proving smallness of the anisotropy term.
\end{enumerate}
{We next review the main arguments.}

\paragraph{\emph{1. Approximate transport.}}
By assumption, the perturbation $f$ is supported in $\Dd(z, \ell)$ and thus in a square $\square(z, \ell)$ of sidelength $\ell$. Using \cite[Lemma 4.8]{serfaty2020gaussian} together with \cite[(4.31)]{serfaty2020gaussian} and applying our assumptions \eqref{assum-Main-non-bound-m} and \eqref{assum-Main-non-bound-f}, we can find a vector field $\psi$, supported on $\square(z, \ell)$, such that
\begin{equation}
\label{assum_psi}
- \div(\rm \psi) = f, \quad |\psi|_\kk \leq \bar \Cc \Cc \ell^{-1 - \kk}, \text{ for } \kk = 0, 1, 2, 3,
\end{equation}
where $\Cc$ is as in the statement and $\bar \Cc$ is a universal constant.

For $s \in \R$, we define a map $\Phi_s$ and another probability measure $\tmms$ by:
\begin{equation}
\label{approx_objects}
\Phi_s := \id + s \psi, \quad \tmms := \Phi_s \# \mm.
\end{equation}
\corT{Here and below, we write $\#$ to denote the push-forward of a measure.}
In view of \eqref{assum_psi} we can guarantee that $|\Phi_s - \id|_\1 \leq \hal$ as long as the parameter $s$ is chosen smaller than $\frac{1}{\Cc} \ell^2$ for some $\Cc$ large enough, a condition which is implied by the stronger constraint \eqref{eq:condi_s} ({recall that $N \ell^2$ is always larger than $1$}). We think of $\Phi_s$ as an \emph{approximate transport map}, pushing $\mm$ forward not quite onto $\mms$ (whose density is $\rm + s f$) but rather on $\tmms$. The following lemma quantifies the error in terms of partition functions.

\begin{lemma}
\label{lem:ApproxError}
The partition functions associated to $\tmms$ and $\mms$ are close:
\begin{equation*}
\left|\log \KNbeta(\tmms) - \log \KNbeta(\mms) \right| = O\left(s^2 N \ell^{-2} \right),
\end{equation*}
and so are their relative entropies:
\begin{equation*}
\left| \EE(\trms) - \EE(\rms) \right| = O\left(s^2 N \ell^{-2}\right),
\end{equation*}
with implicit constants depending only on \eqref{assum_psi} and $\beta$.
\end{lemma}
\begin{proof}
The first point follows from combining \cite[Lemma 5.1]{serfaty2020gaussian}, which bounds $|\trms - \rms|_{\kk}$ for $\kk = 1, 2$ in terms of the norms of $\psi$ (controlled in \eqref{assum_psi}) and $\mm$ (controlled by assumption), and \cite[Lemma 4.9]{serfaty2020gaussian}, which states a direct comparison of the partition functions in terms of $\tmms - \mms$. To prove the second point, we use again \cite[Lemma 5.1]{serfaty2020gaussian}, which bounds $|\trms - \rms|_{\0}$ by $O(s^2 \ell^{-4})$, and plug that estimate into the definition of $\EE$.
\end{proof}

\paragraph{\emph{2. Comparison of partition functions along a transport.}}
\newcommand{\ErrorF}{\mathsf{ErrorF}}
The so-called ``anisotropy term'' was introduced in \cite{LebSerCLT}, cf.~also the ``angle term'' in \cite[Section 8]{bauerschmidt2019two}. We refer to \cite[Section~4]{serfaty2020gaussian} for a careful study of its properties. It corresponds to the first-order correction to the energy when one pushes \emph{both the configuration and the background measure} by a small perturbation of the identity map. Here we will not go into the details and we treat the anisotropy as a black box. The key estimates that we need to import are contained in the following lemma. Let $s, \psi, \Phi_s, \tmms$ be as above (in particular $\psi$ is  supported on $\square(z, \ell)$) and recall that $\EnerPts(z, \ell)$ (defined in \eqref{def:EnerPts}) controls both the energy and the number of points at scale $\ell$ near $z$, and is typically of order $N\ell^2$.

\begin{lemma}
\label{lem:AAAni}
There exists a term $\AA[\psi, \mm, \XN]$ (independent of $s$) satisfying:
\begin{equation}
\label{boundAni}
\AA[\psi, \mm, \XN] = \ell^{-2}  O\left(\EnerPts(z,\ell)\right),
\end{equation}
and such that provided $|s|$ is smaller than $\ell^2$:
\begin{equation}
\label{Ani_ordre_2}
\F\left( \Phi_s(\XN), \tmms \right) = \F(\XN, \mm) + s \Ani[\psi, \XN, \mm] + s^2 \ErrorF,
\end{equation}
with a ``second order'' error term $\ErrorF$ bounded by:
\begin{equation}
\label{ErrorF}
\ErrorF = O \left( \ell^{-4} (1 + \log(\ell N^{1/2}))  \EnerPts(z,\ell)\right).
\end{equation}
with implicit constants in \eqref{boundAni} and \eqref{ErrorF} depending only {on the constant $\Cc$ in \eqref{assum-Main-non-bound-m}, \eqref{assum-Main-non-bound-f}}.
\end{lemma}

\begin{proof}[Proof of Lemma \ref{lem:AAAni}]
We apply \cite[Prop. 4.2]{serfaty2020gaussian} to the vector field $\psi$ chosen above. The main task is to check that {\cite[(4.7)]{serfaty2020gaussian}}, which bounds the second derivative of the energy along a transport, can itself be controlled by our $\ErrorF$.
\begin{itemize}
	\item We use \eqref{assum_psi} to bound $\psi$ and its derivatives.
	\item Using our assumptions on $\rm$ and $f$ we can control the derivatives of $\rm_s$ by:
	\begin{equation*}
	|\rm_s|_\kk \leq |\rm|_\kk + s |f|_\kk \leq \Cc \ell^{-\kk} + s \Cc \ell^{-\kk +2},
	\end{equation*}
	but we are taking $|s| \ll \ell^2$ (see \eqref{eq:condi_s}  and \eqref{condiell}) so we can replace $|\rm_s|_\kk$ by $\Cc \ell^{-\kk}$ for $\kk = 1, 2$.
In particular, $N^{-\kk/2} |\rm_s|_\kk \preceq 1$   for $\kk = 1, 2$.
\end{itemize}
Combining these observations and studying \cite[(4.8)]{serfaty2020gaussian}, we are left with:
\begin{equation*}
\left(\ell^{-4} \Cc + \ell^{-5} N^{-1/2} \log( \ell N^{1/2} ) \Cc\right) \times \log ( \ell N^{1/2} ) \EnerPts(z, \ell) +
\Cc \ell^{-5} N^{-1/2} \EnerPts(z, \ell).
\end{equation*}
Simplifying a bit and using that $\frac{1}{\ell^2 N} \leq 1$, we can indeed take $\ErrorF$ as in \eqref{ErrorF}.
\end{proof}

For convenience, we denote by $\Ani$ the quantity:
\begin{equation}
\label{eq:defAni}
\Ani[\psi, \mm, \XN] := \AA[\psi, \mm, \XN] - \frac{1}{4} \sum_{i=1}^N \div\ \psi(x_i).
\end{equation}
We can now state the conclusion of this paragraph:
\begin{lemma}[Comparison of partition functions along a transport]
\label{lem:transport}
Recall the notation $\Pbeta_{N, \mm}$ from \eqref{def:PNbetamu}.
We have:
\begin{multline}
\label{compa_transport}
\log \frac{\KNbeta(\tmms)}{\KNbeta(\mm)} = \left(\frac{\beta}{4} - 1 \right) N \left( \EE(\trms) - \EE(\rm) \right) \\ + \log \Esp_{\Pbeta_{N, \mm}} \left[\exp\left(s \Ani[\psi, \XN, \mm] + s^2 \ErrorF \right) \right].	
\end{multline}
\end{lemma}
\begin{proof}[Proof of Lemma \ref{lem:transport}] It follows the same steps as \cite[Prop. 4.3]{LebSerCLT}. In short:
\begin{enumerate}
	\item We change variables {by $\Phi_s$} in the integral defining $\KNbeta(\tmms)$ and use Lemma \ref{lem:AAAni} to expand the interaction energy, hence the terms $\AA$ and $\ErrorF$ appear in the exponential. We replace $\AA$ by $\Ani$ up to the correction mentioned in \eqref{eq:defAni}.
	\item The Jacobian of this change of variables is $\exp\left(\sum_{i=1}^N \log \det \D \Phi_s(x_i)\right)$. {The exponent} can be related to the difference of entropies between $\mms$ and $\mm$, up to a fluctuation term (the fluctuations of $x \mapsto \log \det \D \Phi_s(x)$), which itself is bounded using Lemma \ref{lem_fluct_uniform} and is negligible compared to the estimate \eqref{ErrorF} for
	  the ``second order'' error $\ErrorF$.
	\item As in (2), the difference between $\Ani$ and $\AA$  in \eqref{eq:defAni} also happens to be equal to the difference of relative entropies (up to some multiplicative factor) plus an error term that is negligible compared to $\ErrorF$.
\end{enumerate}
\end{proof}

\paragraph{\emph{3. Relative expansion of partition functions.}}
The following is contained in \cite[Prop. 6.4]{serfaty2020gaussian} \corT{(see Section \ref{sec:tech_rem} below for a technical comment)}.
\begin{lemma}[Relative expansion of partition functions]
\label{lem:relative_expansion}
Let $(z, \ell)$ such that \eqref{condiell} holds, and let $\mm, \tmm$ be two probability measures on $\Dd$ such that $\rm = \trm$ outside $\Dd(z, \ell)$. Assume that $\rm, \trm$ are of class $C^1$ and satisfy $\hal \leq \rm, \trm \leq 2$ on $\Dd$. Assume also that $|\rm|_\1 + |\trm|_\1 \leq \Cc \ell^{-1}$ for some constant $\Cc$. Then:
\begin{equation}
\label{compa_relat}
\log \frac{\KNbeta(\tmm)}{\KNbeta(\mm)} = \left(\frac{\beta}{4} - 1 \right) N \left( \EE(\trm) - \EE(\rm) \right) + O \left( N \ell^2 \right)^\hal \left(1 + \log \left(N \ell^2\right) \right)^\hal,
\end{equation}
with an implicit constant depending only on $\Cc$ and $\beta$.
\end{lemma}

\paragraph{\emph{4. Serfaty's trick.}}
It is hard to prove a bound on $\Ani$ that is better than \eqref{boundAni} and holds \emph{configuration-wise}. However one can improve the control on $\Ani$ \emph{in exponential moments}, using the following trick.
Recall the assumption  \eqref{condiell}.

\begin{lemma}[Smallness of the anisotropy]
\label{lem:AniSmall}
There exists $\Cc'$ depending only on {on the constant $\Cc$ in \eqref{assum-Main-non-bound-m}, \eqref{assum-Main-non-bound-f}} such that if $|s| \leq \frac{1}{\Cc'} \ell^{3/2} N^{-1/4}$, we have:
\begin{equation}
\label{ani_small}
\log \Esp_{\Pbeta_{N, \mm}} \left[\exp\left(s \Ani[\psi, \XN, \mm]\right) \right] = O\left(s N^{3/4} \ell^{-1/2} \left(1 + \log\left(N \ell^2 \right) \right)^\hal \right) ,
\end{equation}
with an implicit constant depending only on $\Cc'$ and $\beta$.
\end{lemma}
\begin{proof}[Proof of Lemma \ref{lem:AniSmall}]
We first take $s = \sstar := \frac{1}{\Cc'} \ell^{3/2} N^{-1/4}$ for some large enough constant $\Cc'$ (if $\ell$ is of order $N^{-1/2}$ then $\ell^2$ and $\ell^{3/2} N^{-1/4}$ are comparable, thus we need to divide by $\Cc'$ large enough in order to match our previous assumptions on $s$. For mesoscopic length scales, this is irrelevant). By comparing the two expressions \eqref{compa_transport} and \eqref{compa_relat} and discarding negligible terms one gets:
\begin{equation}
\label{sstar}
\log \Esp_{\Pbeta_{N, \mm}} \left[\exp\left(\sstar \Ani[\psi, \XN, \mm] + \sstar^2 \ErrorF \right) \right] = O\left( \left(N \ell^2 \right)^\hal \left(1 + \log \left(N \ell^2\right) \right)^\hal \right).
\end{equation}
Using the expression \eqref{ErrorF} for $\ErrorF$ and the local law \eqref{loc_laws} we know that:
\begin{equation}
\label{usingErrorF}
\log \Esp_{\Pbeta_{N, \mm}} \left[\exp\left( \sstar^2 \ErrorF \right) \right] \leq O\left(  \ell^{-1} N^{-1/2} \log(\ell N^{1/2}) N \ell^2  \right)   =  O\left( \ell N^{1/2} \log(\ell N^{1/2}) \right)  .
\end{equation}
Combining \eqref{sstar} and \eqref{usingErrorF} and using Cauchy-Schwarz's inequality we deduce that:
\begin{equation*}
\log \Esp_{\Pbeta_{N, \mm}} \left[\exp\left(\hal \sstar \Ani[\psi, \XN, \mm] \right) \right] \leq 
O\left( \ell N^{1/2} \log(\ell N^{1/2}) \right). 
\end{equation*}
Thus for values of $s$ \corG{smaller than $ \hal \sstar$}, we apply Hölder's inequality and obtain \eqref{ani_small}.
\end{proof}

\paragraph{\emph{5. Conclusion of the proof of Lemma \ref{prop:compar1}}}
We compare $\KNbeta(\mms)$ and $\KNbeta(\tmms)$ using Lemma \ref{lem:ApproxError} and then apply Lemma \ref{lem:transport} to compare $\KNbeta(\tmms)$ and $\KNbeta(\mm)$, using Lemma \ref{lem:AniSmall} to control the anisotropy term. Lemma \ref{lem:ApproxError} also allows us to replace $\EE(\trms)$ by $\EE(\rms)$ up to some error. Finally one can check that for $|s| \leq \sstar$ the dominant error term is the one coming from \eqref{ani_small}, which yields \eqref{eq:compar1}. \qed


\subsubsection{Proof of Lemma \ref{lem:compar_nice_bound}}
Compared to the previous proof, we dispense with Step 1 as one can easily find an exact transport, as well as Steps 3 and 4 because we are not aiming for precise estimates on the anisotropy. \\
Since the reference measure is $\mm_0$ and since $f$ has radial symmetry, it is easy to construct a bijective, ``radial rearrangement'' map $\Phi_s : \Dd \to \Dd$ that pushes $\mm_0$ onto $\mm_0 + s f$ and can be written as $\Phi_s = \id + \alpha_s(x) x$ on $\Dd$, for some radial function $\alpha_s$ whose derivative is bounded on $\Dd$ by $s \Cc''$ for some constant $\Cc''$ depending only on $\Cc$ and $r$. We can always extend $\alpha_s$ into a $C^1$, compactly supported function on $\Dd_2$ with $|\alpha_s| \leq \hal, |\alpha_s|_1 \leq s \Cc''$.\\
Applying \cite[Prop 4.2]{serfaty2020gaussian} (note that now, compared to the proof of Proposition \ref{prop:compar1}, we are only able to use - and in fact only need - results where the vector field $\psi$ is simply assumed to be $C^1$, see in particular \cite[(4.6)]{serfaty2020gaussian}) to $\Phi_s$ and changing variables as in the previous proof, we obtain (see also \cite[(4.3)]{serfaty2020gaussian})
\begin{equation*}
\log \frac{\KNbeta(\mm'_s)}{\KNbeta(\mm_0)} = \log \Esp \left[ \exp\left( s
\Cc'' \EnerPts(\Dd_2) \right) \right],
\end{equation*}
but the number of points in $\Dd_2$ is always $\leq N$ and the (fluctuations of the) global energy 
have exponential moments of order $N$ (this “global law” follows from \eqref{global_law}). We thus get \eqref{conclu-boundary}.
\qed

\subsection{A technical remark about partition functions}
\label{sec:tech_rem}
\newcommand{\KNbetaSer}{\mathrm{K}_N^{\beta, \mathrm{Ser}}}
\corT{We have used above in \eqref{compa_relat} the comparison between partition functions associated to measures that coincide outside a small disk stated in \cite[Prop. 6.4]{serfaty2020gaussian}. It is worth noting that our definition \eqref{def:KNbetamu} of the partition function $\KNbeta(\mm)$ associated to a probability measure $\mm$ on $\Dd$ differs slightly from the one in \cite{serfaty2020gaussian} (which we denote here by $\KNbetaSer$):
\begin{multline*}
\KNbeta(\mm) := \int_{(\R^2)^N} \exp\left(- \beta \left( \F(\XN, \mm) + 2N \sum_{i=1}^N \zeta(x_i) \right) \right) \d \XN, \\
 \KNbetaSer(\mm) := \int_{(\R^2)^N} \exp\left(- \beta \left( \F(\XN, \mm) \right) \right) \prod_{i=1}^N \rm(x_i) \d  \XN.
\end{multline*}
This is due to the fact that we work with standard equilibrium measures $\mm$ whereas \cite{serfaty2020gaussian} prefers to work with the so-called “thermal equilibrium measure” - which allows to treat various temperature regimes in a unified fashion, a feature which is not relevant for us here. Let us briefly investigate the consequence of this difference in the definition.\\
Assume (as will always be the case for us) that $\mm$ is supported on the unit disk and that its density is $C^1$ on $\Dd_1$ and is equal to $1$ inside $\Dd_1$ near $\partial \Dd_1$. First, assume moreover that we have a “perfect confinement” in the sense of Remark \ref{rem:perfectconfinement} i.e. $\zeta = + \infty$ outside $\Dd_1$. Then both integrals defining  $\KNbeta(\mm)$ and $\KNbetaSer(\mm)$ are effectively integrals on $\Dd_1^N$ and the integrands differ by:
\begin{equation*}
\prod_{i=1}^N \rm(x_i) = \exp\left( \sum_{i=1}^N \log \rm(x_i) \right) = \exp\left(N \int \log \rm(x) \d \mm(x) + \LNmu(\log \rm) \right).
\end{equation*}
In the first term in the exponent, we recognize $N \EE(\rm)$, where $\EE$ is the relative entropy as in Section~\ref{sec:not}. On the other hand, the fluctuation term $\LNmu(\log \rm)$ contributes $O(\sqrt{N})$ because $x \mapsto \log \rm(x)$ is $C^1$ and compactly supported within $\Dd_1$, and we can control fluctuations of Lipschitz test functions thanks to Lemma \ref{lem_fluct_uniform}. We would thus obtain, in that case:
\begin{equation*}
\log \KNbeta(\mm) = \log \KNbetaSer(\mm) - N \EE(\rm) + O(N^\hal),
\end{equation*}
which means that the main difference is one entropy term to be added/substracted from the formulas. If, however, the confinement is not “perfect” then it is not completely clear what the respective influence of $2N \sum_{i=1}^N \zeta(x_i)$ in $\KNbeta$ and $\prod_{i=1}^N \rm(x_i)$ in $\KNbetaSer$ will be.\\
In the proof of \cite[Prop. 6.4]{serfaty2020gaussian}, to compare the partition functions associated to $\mm$ and $\tmm$, the first step is to decompose the partition function into an “outside” part (corresponding to the system \emph{outside} $\Dd(z, \ell)$ where the measures coincide) and an “inside” part (corresponding to $\Dd(z, \ell)$) on which all the analysis takes place. For this “inside” part, the effective confining potential $\zeta$ plays no role, and the only difference in definition is thus again $\prod_{i=1}^n \rm(x_i) = \exp\left( \sum_{i=1}^n \log \rm(x_i) \right).$
We recognize again a relative entropy as the leading order term and we note that, since $\mm$ and $\tmm$ coincide outside $\Dd(z, \ell)$ we have:
\begin{equation*}
\int_{\Dd(z, \ell)} \log \rm \ \d \mm - \int_{\Dd(z, \ell)} \log \trm \ \d \tmm = \EE(\rm) - \EE(\trm).
\end{equation*}
Moreover, since we assume that the density $\rm$ is bounded below by $\hal$ and that $|\rm|_{\1}$ is of order $\ell^{-1}$ (see the statement of Lemma \ref{lem:relative_expansion}) we can ensure that the fluctuation term within $\Dd(z, \ell)$ is of order $\left(N\ell^2\right)^\hal$ in exponential moments. In conclusion: 
\begin{equation*}
\log \frac{\KNbeta(\tmm)}{\KNbeta(\mm)} = \log \frac{\KNbetaSer(\tmm)}{\KNbetaSer(\mm)} + N \left( \EE(\rm) - \EE(\trm) \right) + O\left(\left(N\ell^2\right)^\hal\right),
\end{equation*}
which allows us to use \cite[Prop. 6.4]{serfaty2020gaussian} with our convention for the partition function, up to adding/substracting one entropy term in the formulas - which explains why we write $\left(\frac{\beta}{4} - 1 \right) N \left( \EE(\trm) - \EE(\rm) \right)$ in \eqref{compa_relat} instead of $\frac{\beta}{4}  N \left( \EE(\trm) - \EE(\rm) \right)$ in \cite[Prop. 6.4]{serfaty2020gaussian}.
}

\subsection{Proof of Proposition \ref{prop:compar2}}
\label{sec:ProofCompar2}
\begin{proof}[Proof of Proposition \ref{prop:compar2}]
We introduce a sequence of intermediate length scales $\ellA = : \ell_0 < \dots < \ell_n : = \ell :=   \min\left(\ellB, \frac{1}{4}\left(1 - \dist(\zA, \partial \Dd)\right) \right)$ by defining $c, n$ as:
\begin{equation*}
n = \floor{\log \left( \ell / \ellA \right)}, \quad \log c := \frac{1}{n} \log \left( \ell/\ellA \right),
\end{equation*}
and by setting $\ell_k := c^k \ellA$ for $k = 0, \dots, n$. Note that $c \in [1,2]$  and $(\zA, \ell_i)$ satisfy \eqref{condiell} for all $i\le n$ even if the larger length scale $\ellB$ is macroscopic. Moreover, $\ellB$ and $\ell$ are always comparable.

For each $i = 0, \dots, n$ we fix a smooth cut-off function $\chi^{(i)}$ such that:
\begin{equation*}
\chi^{(i)} \equiv 1 \text{ on } \Dd(\zA, \ell_i), \quad \chi^{(i)} \equiv 0 \text{ outside } \Dd(\zA, 2 \ell_i), \quad |\chi^{(i)}|_{\kk} \leq 100 \ell_i^{-\kk} \text{ for } \kk = 1, 2, 3,
\end{equation*}
and we define $\mmsi$ as the probability measure with the following density:
\begin{equation}
\label{def:rmsi}
\rmsi := \rm + s \fB + s  \left(\int_{\R^2} \fA\right) \frac{\chi^{(i)}}{\int_{\R^2} \chi^{(i)}}.
\end{equation}
In particular, in case $\int \fA =0$, this construction is not relevant and we can directly apply
Proposition \ref{prop:compar1} twice to compare ${\KNbeta(\mms)}$ to ${\KNbeta(\mm+s \fB)}$ and then ${\KNbeta(\mm+s \fB)}$ to ${\KNbeta(\mm)}$.
This is also the case if the scales $\ellA$ and $\ellB$ are comparable, which corresponds to the case where $n$ is independent of $N$.

Using our assumptions on $\mm, \fA, \fB, \chi^{(i)}$ and the fact that \eqref{condi:s2} and \eqref{condiell} imply that $|s|$ is of order at most $\ellA^2$ (the minimal scale in this problem), we obtain that on $\Dd$ the density $\rmsi$ satisfies:
\begin{equation}
\label{reg_msi}
\frac{1}{2} \leq \rmsi \leq 2, \quad |\rmsi|_\kk \preceq \Cc \ell_i^{-\kk} \text{ for } \kk = 1, 2, 3.
\end{equation}

We decompose the ratio of partition functions as :
\begin{equation*}
\frac{\KNbeta(\mms)}{\KNbeta(\mm)} = \frac{\KNbeta(\mms)}{\KNbeta(\mms^{(0)})} \times \prod_{i=0}^{n-1} \frac{\KNbeta(\mmsi)}{\KNbeta(\mmsip)} \times \frac{\KNbeta(\mms^{(n)})}{\KNbeta(\mm)}.
\end{equation*}
For each $i = 0, \dots, n-1$, in view of \eqref{def:rmsi} we may write:
\begin{equation*}
\rmsi = \rmsip + s \fii, \text{ with } \fii := s \left(\int_{\R^2} \fA\right) \left( \frac{\chi^{(i)}}{\int_{\R^2} \chi^{(i)}} - \frac{\chi^{(i+1)}}{\int_{\R^2} \chi^{(i+1)}} \right).
\end{equation*}
Since $\int \fA$ is of order $1$ (this follows from \eqref{condclt}), while $\int \chi^{(i)}$ is of order $\ell_i^2$, the perturbation $\fii$ satisfies:
\begin{equation}
\label{reg_fii}
|\fii|_{\kk} \preceq \ell_i^{-(\kk+2)} \text{ for } \kk = 0, 1, 2.
\end{equation}
In view of \eqref{reg_msi} and \eqref{reg_fii} we can apply Proposition \ref{prop:compar1} with the length scale chosen as $\ell_i$ and we obtain:
\begin{equation*}
\log \frac{\KNbeta(\mmsip)}{\KNbeta(\mmsi)}\\
= N \left(\frac{\beta}{4} - 1 \right) \left( \EE(\rmsip) - \EE(\rmsi) \right) +  O \left( s N^{3/4} \ell_i^{-1/2} \left(1 + \log\left(N \ell_i^2 \right) \right)^\hal\right),
\end{equation*}
The same reasoning applies for $\frac{\KNbeta(\mms)}{\KNbeta(\mm^{(0)})}$ (which corresponds to the smallest length scale $\ellA$) and for $\frac{\KNbeta(\mm_s^{(n)})}{\KNbeta(\mm)}$ (which corresponds to the largest length scale $\ell$ --  either equal to $\ellB$ or  of order $1$). Summing the contributions, we obtain:
\begin{equation*}
\log \frac{\KNbeta(\mms)}{\KNbeta(\mm)} = N \left(\frac{\beta}{4} - 1 \right) \left( \EE(\rms) - \EE(\rm) \right) + O\left( s \sum_{i=0}^n  N^{3/4} \ell_i^{-1/2} \left(1 + \log\left(N \ell_i^2 \right) \right)^\hal \right).
\end{equation*}
The sum over dyadic scales can be compared to an integral and we get:
\begin{equation*}
\sum_{i=0}^n  N^{3/4} \ell_i^{-1/2} \left(1 + \log\left(N \ell_i^2 \right) \right)^\hal = O \left( N^{3/4} \ellA^{-1/2} \left(1 + \log(N\ellA^2) \right)^\hal \right),
\end{equation*}
which yields \eqref{eq:compar2} as claimed.
\end{proof}

\subsection{Proof of Proposition \ref{prop:expg}}
\label{sec:proofexpo_hmu}

Let $\hmu$ denote the electrostatic potential generated by the uniform background $\mm_0$, namely:
\begin{equation*}
\hmu : z \mapsto \int_{\Dd} \log|z-x| \d \mm_0(x) = \int_{|x| \leq 1} \log|z-x| \frac{\d x}{\pi},
\end{equation*}
a quantity that already appeared in the definition \eqref{def:PotXN} of $\Pot$. An explicit computation gives the following expression
\begin{equation}
\label{def:hmu}
\hmu : z \mapsto \begin{cases}
\log |z| & \text{if } |z| \geq 1 \\
-\frac{1 - |z|^2}{2} & \text{ if } |z| \leq 1.
\end{cases}
\end{equation}
The potential $\hmu$ satisfies Poisson's equation $\Delta \hmu = 2\pi \mm_0$. In particular, $\Delta \hmu$ does \emph{not} have total mass $0$. Yet, using a simple trick that is totally unrelated to the methods of \cite{LebSerCLT,serfaty2020gaussian}, we are able to control the size of $\LN(\hmu)$ as expressed in the following proposition:
\begin{proposition}[Exponential moments of $\hmu$]
\label{prop:expo_hmu}
We have, for $|t| \leq \frac{N \beta}{100}$:
\begin{equation*}
\log \Esp \left[ \exp\left( t \LN(\hmu) \right) \right] = O(t + t^2),
\end{equation*}
with an implicit constant depending only on $\beta$.
\end{proposition}
{Proposition \ref{prop:expo_hmu}} gives us the crucial freedom to
substract a multiple of $\hmu$ from a given test function and to cancel out the masses of their Laplacians, while making a (typically) bounded error on the size of the linear statistics. We postpone the proof of Proposition \ref{prop:expo_hmu} for now and deduce from it a generalization of Proposition \ref{prop:expg}.

\begin{corollary} \label{cor:expo_hmu}
Let $\chi$ be a $C^2$, radially symmetric function which is compactly supported in $\Dd_{r}$ for some $r< 1$, such that $|\chi|_{\kk} \leq \Cc$ for $\kk = 0, 1, 2$.
Let $\g$ be a solution of Poisson's equation $\Delta \g =  2\pi\chi$. Then, one has for $|t| \ll N$,
\[
\log\Esp [ e^{t \LN(\g)}]  = O(t + t^2)
\]
where the implied constant depends only on $\beta$.
\end{corollary}

\begin{proof}
Let $c=\int \chi$ which is of order 1 (by assumptions) and let $\varphi = \g - c\hmu$.
By Lemmas~\ref{lem:rewrite_Laplace} and~\ref{lem:compar_nice_bound}, one has for $|t| \ll N$,
\[
\log\Esp [ e^{t \LN(\varphi)}]  = O(t + t^2) .
\]
This estimate is comparable to that of Lemma~\ref{prop:expo_hmu}, thus the claim follows directly from H\"older's inequality.
\end{proof}

\begin{proof}[Proof of Proposition \ref{prop:expo_hmu}]
This relies on the following two claims:
\begin{claim}
\label{claim_zeta}
For $|t| \leq \beta N$, we have:
\begin{equation}
\label{eq:zeta_does_not_matter}
\log \Esp\left[ e^{t  \sum_{i=1}^N \zeta(x_i) } \right] =  o(t).
\end{equation}
\end{claim}
\begin{proof}
This follows from the analysis of \cite{LebSerCLT}. We return to the notation of \eqref{def:KNbeta}, \eqref{def:KNbetamu} for partition functions and make them more explicit
by writing:
\begin{equation*}
\KNbeta(\mm_0, \zeta) := \int_{\left(\R^2\right)^N} \exp\left( - \beta \left( \F(\XN, \mm_0) + 2N \sum_{i=1}^N \zeta(x_i) \right) \right) \d\XN,
\end{equation*}
where we explicitly keep track of both the background measure (as in \eqref{def:KNbetamu}) \emph{and} the ``effective confining potential'' $\zeta$. By \cite[Corollary 1.1]{LebSerCLT} or \cite[(4.12)]{LebSerCLT} one finds an expansion for $\log \KNbeta(\mm_0, \zeta)$ up to order $o(N)$ which \emph{does not depend on $\zeta$} (see \cite[Remark 4.3]{LebSerCLT}). Thus in particular, taking $t = \pm \beta N$ in the left-hand side of \eqref{eq:zeta_does_not_matter} we get:
\begin{equation*}
\log \Esp\left[ e^{ \pm \beta N  \sum_{i=1}^N \zeta(x_i) } \right] = \log \KNbeta(\mm_0, \zeta \mp \frac{1}{2} \zeta) - \log \KNbeta(\mm_0, \zeta) = o(N).
\end{equation*}
The claim follows from H\"{o}lder's inequality.
\end{proof}

\begin{claim}
\label{claim_scalin}
For $|t| \leq \frac{N \beta}{2}$, we have:
\begin{equation}
\label{scaling_xi}
\Esp\left[ e^{t \sum_{i=1}^N |x_i|^2 } \right] = \exp\left(-\log\left( 1 - \frac{t}{\beta N} \right) \left(\frac{\beta N(N-1)}{4} + N \right)  \right).
\end{equation}
\end{claim}
\begin{proof}
It follows by a scaling argument using the equivalent expression \eqref{eq:Pnbetav2} for the joint law of the particles.
\end{proof}

We may now prove Proposition \ref{prop:expo_hmu}. On the one hand, using the expression \eqref{def:hmu} of $\hmu$ and an elementary computation we get:
\begin{equation*}
\int \hmu(x)  \d \mm_0(x) = \int \frac{|x|^2 - 1}{2} \d \mm_0(x) = - \frac{1}{4}.
\end{equation*}
On the other hand, a Taylor's expansion of \eqref{scaling_xi} gives (for $\frac{|t|}{\beta N}$ small):
\begin{equation*}
\Esp\left[ e^{t \sum_{i=1}^N |x_i|^2} \right] = \exp\left( \frac{t N}{4} + O(t + t^2) \right),
\end{equation*}
and we deduce that:
\begin{multline}
\label{esp_presque_hmu}
\Esp\left[ e^{ t \left( \sum_{i=1}^N \frac{|x_i|^2 - 1}{2} \right) - t N \int \hmu(x)  \d \mm_0(x) } \right] = \Esp\left[ e^{ \frac{t}{2}  \sum_{i=1}^N |x_i|^2  - \frac{tN}{2} + \frac{tN}{4} } \right] \\
= \Esp\left[ e^{ \hal \left( \frac{t N}{4} + O(t + t^2) \right) - \frac{t N}{8} } \right]  = \Esp\left[ e^{ O(t + t^2) } \right].
\end{multline}
Since the expression of $\hmu(x)$ is not always $\frac{|x|^2 - 1}{2}$ but rather $\frac{|x|^2 - 1}{2} - \zeta(x)$ in general (compare \eqref{def:hmu} with \eqref{def:zeta}), we can write:
\begin{equation*}
\Esp\left[ e^{ t \LN(\hmu)} \right] = \Esp\left[ e^{ t \left( \sum_{i=1}^N \frac{|x_i|^2 - 1}{2} \right) - t N \int \hmu(x)  \d \mm_0(x) - t \sum_{i=1}^N \zeta(x_i)} \right].
\end{equation*}
Using Cauchy-Schwarz's inequality combined with \eqref{eq:zeta_does_not_matter} and \eqref{esp_presque_hmu} we prove Proposition \ref{prop:expo_hmu}.
\end{proof}

\bibliographystyle{alpha}
\bibliography{LLNMaxPotbib}

\end{document}